\documentclass[USenglish]{article}	
\usepackage[utf8]{inputenc}				%(only for the pdftex engine)
\usepackage[big,online]{dgruyter}	%values: small,big | online,print,work
\usepackage{lmodern} 
\usepackage{microtype}
\usepackage[numbers,square,sort&compress]{natbib}
\usepackage{tikz}

\def\P{{\mathbb P}}
\def\R{{\mathbb R}}

\def\N{{\mathbb N}}

\def\T{{\mathbb T}}

\def\TT{{\mathcal T}}

\newcommand{\dist}[3][]{{\rm d\!l}[\ifthenelse{\equal{#1}{}}{}{#1;}#2,#3]}
\newcommand{\eff}[3][]{{\rm osc}\ifthenelse{\equal{#1}{}}{}{_{#1}}(#2;#3)}

\def\norm#1#2{\|#1\|_{#2}}

\def\set#1#2{\big\{#1\,:\,#2\big\}}

\def\eps{\varepsilon}
\def\normL2#1#2{\|#1\|_{L^2(#2)}}
\theoremstyle{dgthm}
\newtheorem{theorem}{Theorem}

\newtheorem{lemma}[theorem]{Lemma}

\newtheorem{algorithm}[theorem]{Algorithm}

\theoremstyle{dgdef}
\newtheorem{definition}{Definition}

\newtheorem{remark}{Remark}

\makeatletter
\def\BState{\State\hskip-\ALG@thistlm}
\makeatother

\begin{document}

%%%--------------------------------------------%%%
	\articletype{Research Article}
	\received{Month	DD, YYYY}
	\revised{Month	DD, YYYY}
  \accepted{Month	DD, YYYY}
  \journalname{De~Gruyter~Journal}
  \journalyear{YYYY}
  \journalvolume{XX}
  \journalissue{X}
  \startpage{1}
  \aop
  \DOI{10.1515/sample-YYYY-XXXX}
%%%--------------------------------------------%%%

\title{Adaptive Image Compression via Optimal Mesh Refinement}
\runningtitle{Adaptive Image Compression}

\author[1]{Michael Feischl}
%\ use * to mark the author as the corresponding author
\author*[2]{Hubert Hackl}
\runningauthor{M.~Feischl and H.~Hackl}
\affil[1]{\protect\raggedright 
TU Wien, Institute of Analysis and Scientific Computing, Austria, michael.feischl@tuwien.ac.at}
\affil[2]{\protect\raggedright 
TU Wien, Institute of Analysis and Scientific Computing, Austria, hubert.hackl@tuwien.ac.at}
	
%\communicated{...}
%\dedication{...}
	
\abstract{The JPEG algorithm is a defacto standard for image compression. We investigate whether adaptive mesh refinement can be used to optimize the compression ratio and propose a new adaptive image compression algorithm. We prove that it produces a quasi-optimal subdivision grid for a given error norm with high probability. This subdivision can be stored with very little overhead and thus leads to an efficient compression algorithm. 
We demonstrate experimentally, that the new algorithm can achieve better compression ratios than standard JPEG
compression with no visible loss of quality on many images.

The mathematical core of this work shows that Binev's optimal tree approximation algorithm is applicable to image compression with high probability, when we assume small additive Gaussian noise on the pixels of the image. }


\maketitle

\section{Introduction} 
The JPEG algorithm has been one of the standard image compression algorithm for decades. It decomposes any given image into a fixed mesh of $8\times8$ pixel blocks. On each of these blocks the image is transformed via the discrete 2D cosine transform. The resulting sequence of coefficients is then quantized by use of a quantization matrix $Q$ and a rounding step which is designed in such a way that higher frequencies are stored with less precision (or completely omitted). This leads to a significant storage space reduction compared to a bitmap representation of the image. Particularly for photographs, which typically do not exhibit sharp color or brightness changes, this compression method is very effective. The goal of the present work is to generate the background mesh adaptively in an optimal way instead of using a fixed block size of $8\times 8$ in order to improve the compression rate. The mesh refinement algorithms are inspired by adaptive mesh refinement for the approximation of partial differential equations, particularly the adaptive finite element method.

Adaptive mesh refinement for partial differential equations was first mathematically analyzed for the finite element method in the 1980s and 1990s. I.~M.~Babu\v{s}ka initiated a wave of mathematical research by developing the first a~posteriori error estimators~\cite{adaptivefirst} for the finite element method. This opened a new field of numerical analysis and led to tremendous research activity, we only mention the books~\cite{ao00,verf,repinalleinzuhaus}. D\"orfler~\cite{d1996} gave a first rigorous strategy for mesh refinement based on error estimators in 2D more than a decade later and after another five years Morin, Nochetto, and Siebert~\cite{mns} presented the first unconditional convergence proof.

Optimality of the adaptive strategy, i.e., the fact that asymptotically the optimal meshes are generated, was first proven by Babu\v{s}ka and Vogelius in 1D~\cite{oned} and for $d\geq 2$ by Binev, Dahmen, and DeVore~\cite{bdd} in 2004 for adaptive FEM for the Poisson model 
problem. Building on that, Stevenson~\cite{stevenson07} and Cascon, Kreuzer, Nochetto,  and Siebert~\cite{ckns} simplified the algorithms and extended the problem class. We refer to~\cite{axioms,generalqo} for an overview on adaptive mesh refinement. The present work is inspired by adaptive tree approximation, which is an important building block of adaptive algorithms (particularly for time-dependent problems, when mesh coarsening is necessary) developed in~\cite{binev:2004} and extended to the $hp$-setting in~\cite{binev:2018}. Adaptive tree approximation produces meshes on which a given target function can be approximated with near optimal accuracy vs.\ cost ratio. We use this algorithm for the JPEG compression by approximating the target image (which can be seen as a function that is constant on each pixel) on a mesh of rectangles. On each of these rectangles, we perform the standard quantization of the image via the discrete 2D cosine transform as is done for classical JPEG compression.

We propose two algorithms: A first one takes a given image $I$ and produces a vector $S_I$ which encodes the compressed image. A second decoding algorithm produces an approximation $I_{\rm app}$ to the original image by (almost) reversing the compression steps in the vector $S_I$. The algorithms satisfy the following properties:
\begin{enumerate}
	\item $I_{{\rm app}}$ is an approximation of $I$, i.e. there holds
	\begin{equation*}\label{propertyone}
	\|I-I_{{\rm app}}\|<\tau
	\end{equation*}
	for a certain norm and a preset tolerance $\tau$.
	\item The necessary storage space for $S_I$ is (near-)minimal.
	\item The run time for both the encoding and decoding algorithm is (near-)optimal.
\end{enumerate}

Several other adaptive versions of JPEG can be found in the literature: In~\cite{ajpeg1,ajpeg2}, the compression quality on each $8\times 8$ subgrid is chosen adaptively. In~\cite{ajpg3}, an adaptive grid is chosen via a greedy algorithm based on interpolation error and this algorithm is used in~\cite{ajpeg4} to compress images, however, no (near) optimality of the algorithm is proved (or can be expected).

The remainder of this work is structured as follows:
In Section~\ref{sec:adaptive}, we establish the adaptive algorithm for general norms. In Section~\ref{sec:L2}, we apply the algorithm for the $L^2$-norm and derive a near-optimality result. Section~\ref{sec:storage} discusses effective storage of the compressed image and a final Section~\ref{sec:experiments} tests the algorithm on realistic images.

\subsection{The JPEG compression algorithm}\label{sec:basics}
The JPEG compression algorithm was first proposed in the 90s (standardized in~\cite{jpeg}) and is a defacto standard for image compression since.

The compression algorithm transforms the image into the YCbCr color space, which separates the colors into blue-difference chroma, red-difference chroma, and a gamma corrected luminance channel. Since human eyes are more susceptible to changes in luminance than in color, each of the channels is compressed with different quality parameters.
However, the basic algorithm is the same on all of the channels: The image is split into $8\times 8$-blocks of pixels on which the discrete cosine transform is applied resulting in an $8\times 8$ matrix containing the coefficients. Each entry of this matrix is divided by the corresponding entry of the so-called quantization matrix $Q$ of the same size. This matrix is predefined in the JPEG standard and controls the compression ratio of the algorithm. The result of the division is then rounded to the nearest integer, which usually produces many zeros in the higher frequency coefficients. The resulting sequence of integers is further compressed via a lossless runlength encoding algorithm.
Note that the rounding step is the only irreversible step of the compression algorithm and hence the reason for quality loss in JPEG (see, e.g.,~\cite{graham:2018} for further details and references).

\section{The adaptive compression algorithm}\label{sec:adaptive}
In the following, we will develop a mathematical formulation of adaptive JPEG compression.
\subsection{Representation of images}
We suppose the image with $h$ (height) times $w$ (width) pixels is given by a function:
\begin{equation*}
I:\{1,...,h\}\times\{1,...,w\}\to \mathbb{R}^3,
\end{equation*}
where $I(i,j)$ corresponds to the RGB-values of the pixel $(i,j)$, which might be normalized to $[0,1]$ (although the precise encoding does not matter for what follows). The coordinates of the pixels are chosen such that the pixel $(1,1)$ is the top left pixel and $(h,w)$ is the bottom right pixel of the image, hence we can think of the image as a matrix of size $h\times w$ with values in $\mathbb{R}^3$.

The principal idea for obtaining an approximate image $I_{{\rm app}}$ is to divide the image into smaller elements and approximate the image on those via the classical JPEG strategy.

To that end, we define elements $R$ as rectangular subsets of $\{1,\ldots,h\}\times\{1,\ldots,w\}$ with the same aspect ratio as the original image, i.e.
\begin{align*}
    R :=\{ j_R,\ldots, j_R+h_R-1\}\times \{k_R,\ldots,k_R+w_R-1\}\quad\text{with } h_R/w_R=h/w.
\end{align*}
The aspect ratio condition is not strictly necessary, but makes it easier to efficiently store the mesh, as only one dimension and the position of the element have to be stored. Without loss of generality, we may assume that $w$ and $h$ are powers of two in order to ensure that admissible elements exist. A collection of elements $R$ which forms a disjoint union of $\{1,\ldots,h\}\times\{1,\ldots,w\}$ is called a mesh $\TT$.
Corresponding to a mesh $\TT$, we consider the approximate image $I_{{\rm app},\TT}=I_{\rm app}$ computed on that mesh, where we hide the dependency on $\TT$ when clear from the context. By $\TT_0$, we denote the initial mesh which consists of only one element $\TT_0=\{R_0\}$ covering the entire image, i.e., $R_0=\{1,\ldots,h\}\times \{1,\ldots,w\}$.

Finally, by $I|_R$, we denote the restriction of the image onto a given element $R$.
In order to produce an optimally adapted mesh, we consider a simple refinement rule via bisection illustrated in Figure~\ref{fig:refinementrule}. Given a mesh $\TT$, refinement of an element $R\in\TT$ produces a new mesh $\TT'$ such that
\begin{align*}
    \TT'\setminus\TT = \{R_1,\ldots,R_4\},
\end{align*}
i.e., we do not consider mesh closure (since the approximation problems below will be entirely local).
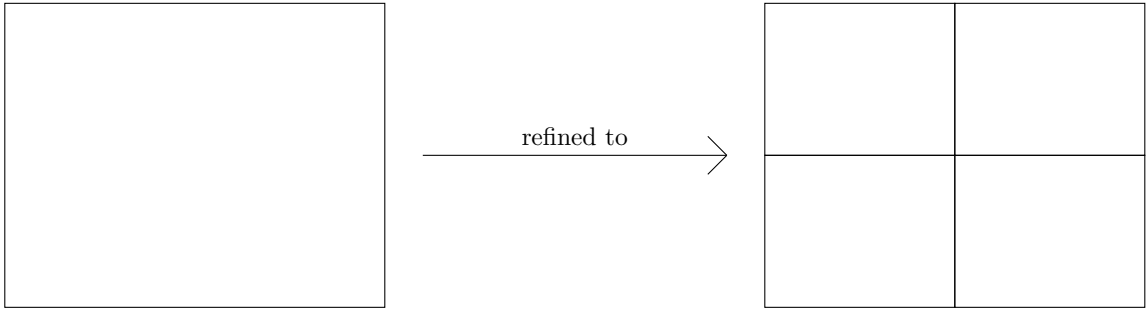
\begin{figure}
	\centering
	\begin{tikzpicture}
	\draw (0,0) rectangle (5,4);
	\draw (5.5,2)--(9.5,2);
	\draw (9.25,1.75)--(9.5,2);
	\draw (9.25,2.25)--(9.5,2);
	\draw (7.5,2.25) node{refined to};
	\draw (10,0) rectangle (12.5,2);
	\draw (10,2) rectangle (12.5,4);
	\draw (12.5,0) rectangle (15,2);
	\draw (12.5,2) rectangle (15,4);	
	\end{tikzpicture}
	\caption{Elements are refined into four children by two bisections of their sides.}\label{fig:refinementrule}
\end{figure}
Note that due to the assumptions on the aspect ratio of the elements, all newly created elements are geometrically similar to $R_0$. 

A mesh $\TT$ is called \textit{coarser} than a mesh $\TT'$ if the mesh $\TT'$ can be obtained by applying refinement steps on $\TT$. In this case we also say $\TT'$ is \textit{finer} than $\TT$ and write $\TT<\TT'$ or $\TT' > \TT$.
	With this notion we further define the sets
	\begin{equation*}
	\T:= \set{\TT}{\TT_0<\TT}
	\end{equation*}
	and
	\begin{equation*}
	\mathbb{T}_n:=\set{\TT\in \T}{\#\TT-\#\TT_0\leq n}.
	\end{equation*}

Note that while two meshes might not be comparable, the refinement rule of our choice still gives \textit{local nestedness}. The following result is well-known and repeated for the convenience of the reader.
\begin{lemma}\label{nestedness}
	Let $\TT',\TT\in \T$. Any element $R\in \TT$ satisfies either $R\subseteq R'$ for some $R'\in\TT'$ or
	\begin{align}\label{eq:nested}
	R= \bigcup_{R'\in\TT'\atop R'\subseteq R}R'.
        \end{align}
	
\end{lemma}
\begin{proof}
If~\eqref{eq:nested} is not true, there must exist some element $R'\in\TT'$ with $R\cap R'\neq \emptyset$ and $R'\not\subseteq R$. Since $R$ and $R'$ are both generated starting from $R_0$, we conclude $R\subseteq R'$.
\end{proof}

\subsection{The adaptive Algorithm}
We follow the strategy of JPEG explained in Section~\ref{sec:basics} and convert the image to the YCbCr color basis, thus resulting in three image components $I^{Y}$, $I^{C_b}$, and $I^{C_r}$.  We downsample the $C_b$ and $C_r$ components to half the image size by computing means of $2\times 2$ blocks of pixels. 
The adaptive algorithm below (Algorithm~\ref{alg:adaptive}) is then applied to each of the components.

The general strategy of the adaptive algorithm will be to start with $\TT_0$ and to iteratively refine the mesh until the compression error in the given norm is sufficiently small. To that end, we require local errors which steer the refinement.

\begin{definition} Let $\|.\|$ be a norm of our choice on the space of images of size $h\times w$, $I$ an image of size $h\times w$ and $\TT$ a mesh. We define:
	
	\begin{itemize}
		\item[(i)] The local error
		\begin{equation}\label{eq:locerr}
		\eta(R):= \|(I-I_{{\rm app},\TT})|_R\|~~~~~\forall R\in \TT
		\end{equation}
		\item[(ii)] The global error
		\begin{equation}\label{eq:globerr}
		E(\TT):= \|(I-I_{{\rm app},\TT})\|
		\end{equation}
		\item[(iii)] The best approximation error for $n$ elements
		\begin{equation*}
		E_n:= \underset{\TT\in \mathbb{T}_n}{\textup{min}}E(\TT).
		\end{equation*}
	\end{itemize}
\end{definition}
As the adaptive algorithm below is essentially a greedy algorithm, these straightforward definitions will not lead to optimal results for most error norms $\norm{\cdot}{}$. It is necessary to consider a modified error first proposed and analyzed in~\cite{binev:2004}.

We denote the modified local error by $\widetilde{\eta}$ and set $\widetilde{\eta}(R_0)=\eta (R_0)$ for the initial element $R_0$. Let $\TT\in\T$ and we assume that $\widetilde{\eta}(R)$ is already defined. We define $\widetilde\eta(R_i)$ for the children $R_i,~i=1,..,4$ of $R$ by
\begin{equation*}
\widetilde{\eta}(R_i)^2:= \frac{\sum_{i=1}^4\eta (R_i)^2}{\eta (R)^2+\widetilde{\eta}(R)^2}\widetilde{\eta}(R)^2.
\end{equation*}
This means $\widetilde{\eta}$ is constant among siblings. The general idea behind this definition is that through the relation to the parent element introduced here, the modified error implicitly depends upon the level of the element, i.e., how many refinements were necessary to produce the element. This forces a non-local behavior of the algorithm which would be impossible with a standard greedy algorithm and turns out to be key to show optimality of the procedure~\cite{binev:2004}.

We are now in the position to formulate the adaptive algorithm.
\begin{algorithm}\label{alg:adaptive}
{\bf Input:} Image component $I$ of size $h\times w$ such that $h,w$ are powers of two and initial mesh $\TT_0$, tolerance $\tau>0$, counter $\ell=0$\\
While $E(\TT_\ell)>\tau$ do:
\begin{enumerate}
    \item Compute modified errors $\widetilde \eta(R)$ for all $R\in\TT_\ell$.
    \item Mark $R\in\TT_\ell$ if $\widetilde \eta(R) = \max_{R'\in\TT_\ell} \widetilde \eta(R')$ and $\min\{h_R,w_R\}\geq 16$.
    \item Refine marked elements to generate $\TT_{\ell+1}$ and set $\ell=\ell+1$.
\end{enumerate}
{\bf Output:} A sequence of meshes $\TT_\ell$.
\end{algorithm}

\begin{remark}
    Note that the algorithm does not produce elements $R$ which are smaller than $8\times 8$, i.e., it does not produce a grid finer than the standard JPEG grid.
\end{remark}
The main goal of the following section will be to prove  near-optimality of Algorithm~\ref{alg:adaptive} by employing the framework developed in~\cite{binev:2004}.
\begin{definition}\label{def:nearopt}
	Let an algorithm compute an approximation of a given image $I$ on some (adaptively refined) mesh $\TT_I\in\T$. We call the algorithm near-optimal (w.r.t. the number of refinements) if there exist positive constants $C_1$, $C_2$ such that for $n,k\in\mathbb{N}$ with $k\leq C_1n$
	\begin{equation*}
	\TT_I\in\T_k\quad\implies \quad E(\TT_I)\leq C_2E_n.
	\end{equation*}
	The constants are universal, i.e., they do not depend on the image $I$.
\end{definition}

\section{Adaptive JPEG compression in the $L^2$-norm}\label{sec:L2}
In this section, we employ Algorithm~\ref{alg:adaptive} with $\norm{\cdot}{}$ denoting the $L^2$-norm weighted with $(wh)^{-1/2}$, i.e., the inverse square root of the total number of pixels.
It is a natural question whether the $L^2$-norm appropriately reflects  human tolerance for quality loss in images (the answer is likely negative, see, e.g.~\cite{quality}) and we refer to Sections~\ref{sec:bvnorm} and Section~\ref{sec:experiments} for further discussion and experiments.
\subsection{Image compression on mesh elements}
The approximate image on a given mesh $\TT$ is produced similarly to the original JPEG algorithm, i.e., we use the discrete cosine transform (DCT) and a quantization Matrix $Q$ to quantize the image. We use a quantization matrix defined in the JPEG standard~\cite{jpeg}:
\begin{equation*}
Q:= \left(\begin{array}{rrrrrrrr}
16 & 11 & 10 & 16 & 24 & 40 & 51 & 61 \\
12 & 12 & 14 & 19 & 26 & 58 & 60 & 55 \\
14 & 13 & 16 & 24 & 40 & 57 & 69 & 56 \\
14 & 17 & 22 & 29 & 51 & 87 & 80 & 62 \\
18 & 22 & 37 & 56 & 68 & 109& 103& 77 \\
24 & 35 & 55 & 64 & 81 & 104& 113& 92 \\
49 & 64 & 78 & 87 & 103& 121& 120& 101\\
72 & 92 & 95 & 98 & 112& 100& 103& 99 \\
\end{array}\right).
\end{equation*}
In the following, we will require some operators to formalize the adaptive compression:
Let $R$ be an element of size $h\times w$. We define a restriction to the top-left $8\times 8$ block of $R$
\begin{subequations}\label{eq:ops}
	\begin{align}
 \begin{split}
	\textup{TL}_R\colon &\mathbb{R}^{h\times w}\to \mathbb{R}^{8\times 8}\\
	&(a_{ij})_{1\leq i \leq h, 1\leq j\leq w}\mapsto (a_{ij})_{1\leq i,j\leq 8}
 \end{split}
	\end{align}
 as well as its right-inverse, an embedding from $\R^{8\times 8}$ into $\R^{h\times w}$ by adding zeros,
	\begin{align}
  \begin{split}
	{\rm TL}_R^{-1}\colon &\mathbb{R}^{8\times 8}\to \mathbb{R}^{h\times w},\quad {\rm TL}_R{\rm TL}_R^{-1} = {\rm Id}\colon \R^{8\times 8}\to\R^{8\times 8}.
 \end{split}
	\end{align}
	
	Furthermore, let $\textup{DCT}_R$ denote the 2D discrete cosine transform on $\mathbb{R}^{h\times w}$ and $\textup{IDCT}_R$ its inverse, i.e.
	\begin{align}
  \begin{split}
	\textup{DCT}_R:\mathbb{R}^{h\times w}&\to \mathbb{R}^{h\times w}\\
	(a_{ij})_{1\leq i \leq h, 1\leq j\leq w}\mapsto (\frac{2}{\sqrt{hw}}C(i)C(j)\sum_{k=0}^{h-1}\sum_{l=0}^{w-1}a_{kl}&\cos \frac{(2k+1)i\pi}{2h}\cos \frac{(2l+1)j\pi}{2w})_{1\leq i \leq h, 1\leq j\leq w}
 \end{split}
	\end{align}
	\begin{align}
  \begin{split}
	\textup{IDCT}_R:\mathbb{R}^{h\times w}&\to \mathbb{R}^{h\times w}\\
	(a_{ij})_{1\leq i \leq h, 1\leq j\leq w}\mapsto (\frac{2}{\sqrt{hw}}\sum_{k=0}^{h-1}\sum_{l=0}^{w-1}C(k)C(l)a_{kl}&\cos \frac{(2k+1)i\pi}{2h}\cos \frac{(2l+1)j\pi}{2w})_{1\leq i \leq h, 1\leq j\leq w},
 \end{split}
	\end{align}
 \end{subequations}
	where $C(n)=1/\sqrt{2}$ if $n=0$ and $C(n)=1$ otherwise.

	Note that these operators solely depend on the size of $R$ and are oblivious to the position of $R$ within the image. Since no elements $R$ with size less then $8\times 8$ will occur in Algorithm~\ref{alg:adaptive}, ${\rm TL}_R$ will always be well-defined.
In the following, let the symbols $\odot$ and $./$ denote entry wise multiplication and division for matrices of the same size.
For a given image $I$ and a mesh $\TT$ we compute an approximate image $I^{\rm final}_{{\rm app},\TT}$ via
	\begin{equation}\label{appi1}
	I_{{\rm app},\TT}^{\rm final}|_R=\textup{IDCT}_R\Big({\rm TL}_R^{-1}\big(Q\odot \textup{round}(\textup{TL}_R(\textup{DCT}_R(I|_R))./Q)\big)\Big)\quad\text{for all } R\in \TT.
	\end{equation}
	However, for the adaptive procedure we will ignore the quantization and therefore only compute the approximate image for the final mesh this way. For a mesh $\TT'< \TT$, which occurs during the adaptive procedure, we compute an approximate image $I_{{\rm app},\TT'}$ by just restricting the image to the first $8\times 8$ DCT coefficients on each element $R\in\TT'$, i.e.,
	\begin{equation} \label{appi2}
	I_{{\rm app},\TT'}|_R=\textup{IDCT}_R\Big({\rm TL}_R^{-1}(\textup{TL}_R(\textup{DCT}(I|_R)))\Big)\quad\text{for all } R\in \TT'.
	\end{equation}
	In the next sections we will only use~\eqref{appi2} and return to~\eqref{appi1} when the goal is to efficiently store and reconstruct the approximate image.
The local and global errors defined in~\eqref{eq:locerr}--\eqref{eq:globerr} thus read
		\begin{equation*}
		\eta(R)^2:= \frac{1}{hw}\sum_{(i,j)\in R} (I_{{\rm app},\TT}(i,j)-I(i,j))^2
		\end{equation*}
	and
		\begin{equation*}
		E(\TT)^2:= \sum_{R'\in \TT}\eta(R')^2=\frac{1}{hw}\sum_{(i,j)\in I}(I_{{\rm app},\TT}(i,j)-I(i,j))^2.
		\end{equation*}

\subsection{The refinement property}
The key for optimality of Algorithm~\ref{alg:adaptive} is the so-called \textit{refinement property} (in literature also called \textit{subadditivity} of the estimator): There exists a universal constant $C_0>0$ such that
\begin{equation}\label{refprop}
\sum_{i=1}^4 \eta (R_i)^2\leq C_0\eta(R)^2~~~~~~\textup{with }R_1,..,R_4\textup{ children of }R\quad\text{for all }R\in\TT\in\T.
\end{equation}
Essentially, the refinement property~\eqref{refprop} states that refinement of the mesh does not significantly increase the approximation error.
Unfortunately, it turns out that this is not true in general in our application as is demonstrated in Figure~\ref{refpropex}.
\begin{figure}
\centering
    \includegraphics[width=0.49\textwidth]{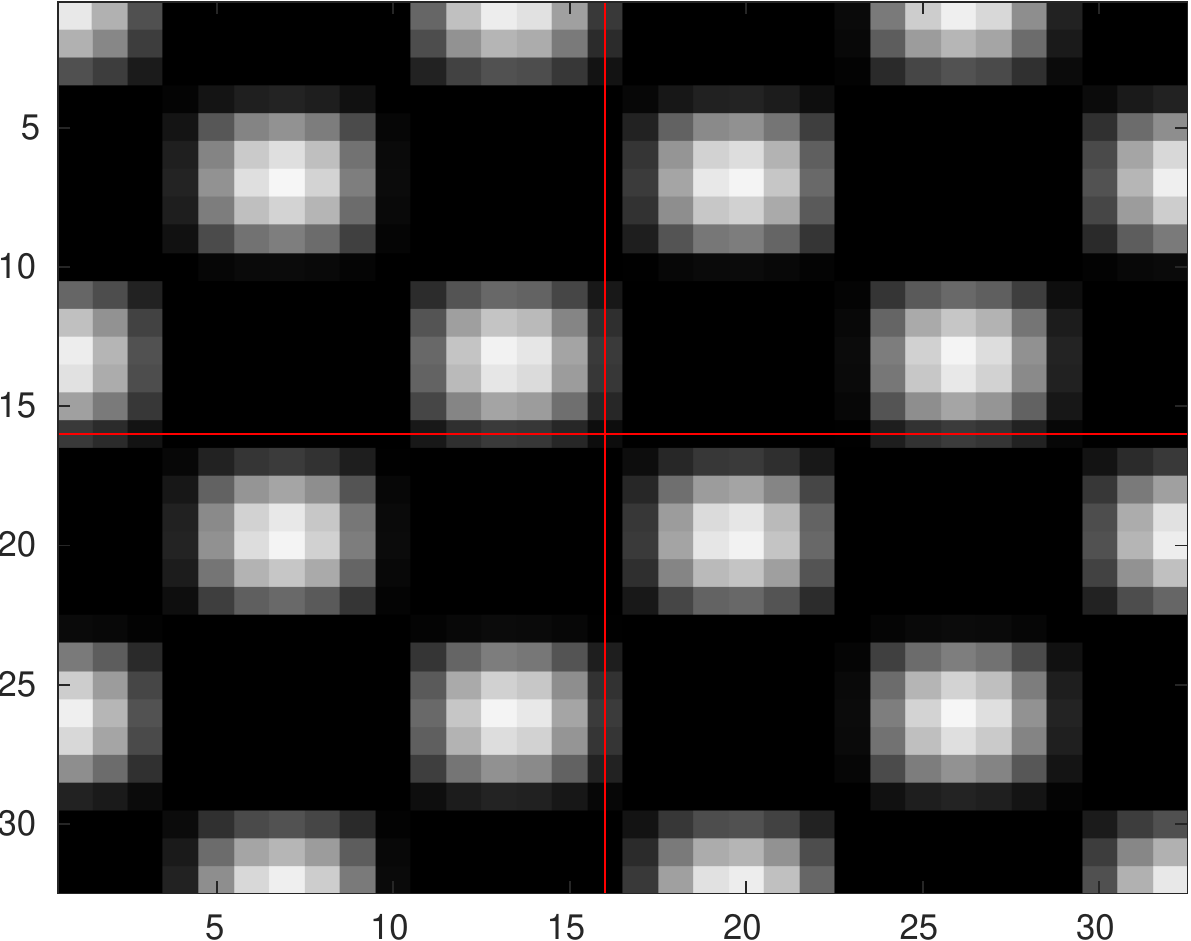}\\
   \includegraphics[width=0.49\textwidth]{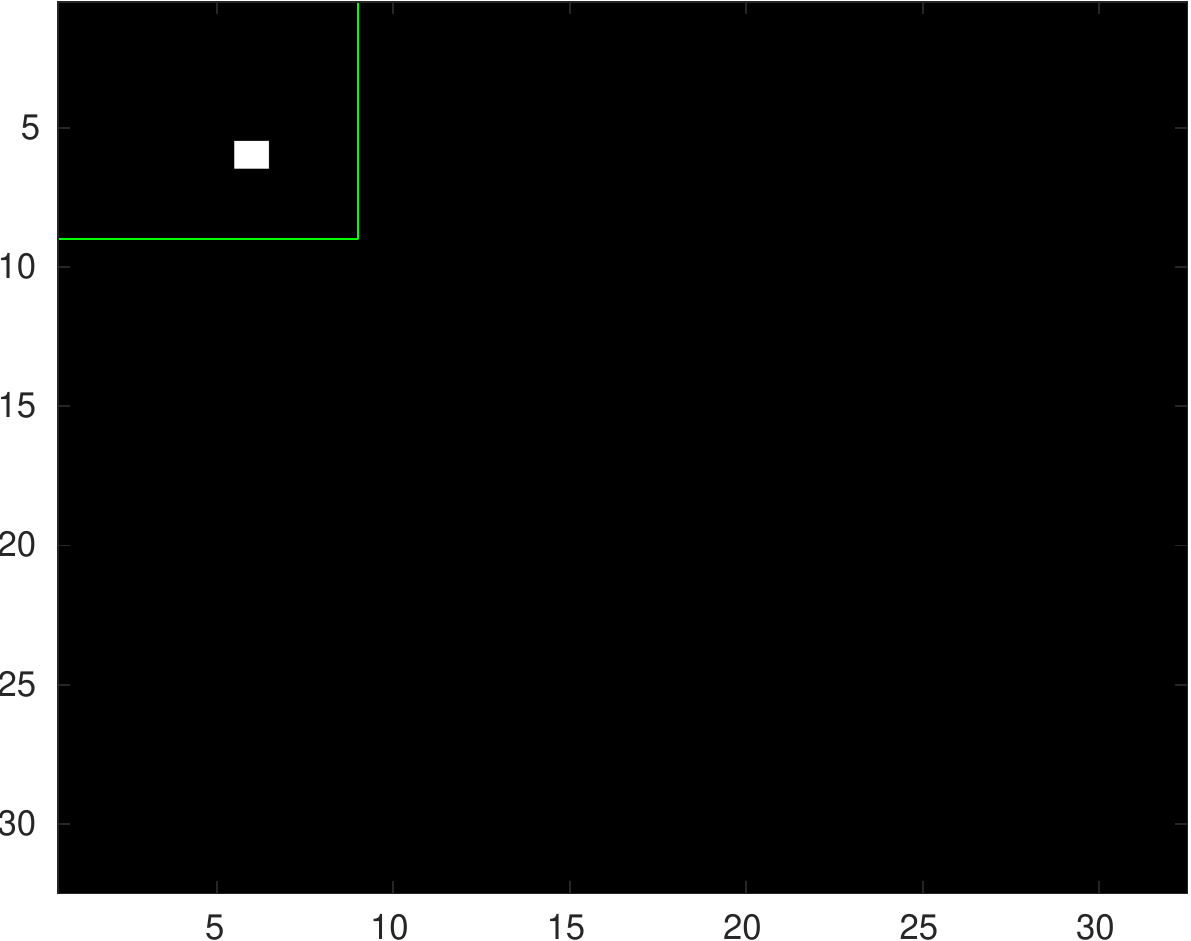}
   \includegraphics[width=0.49\textwidth]{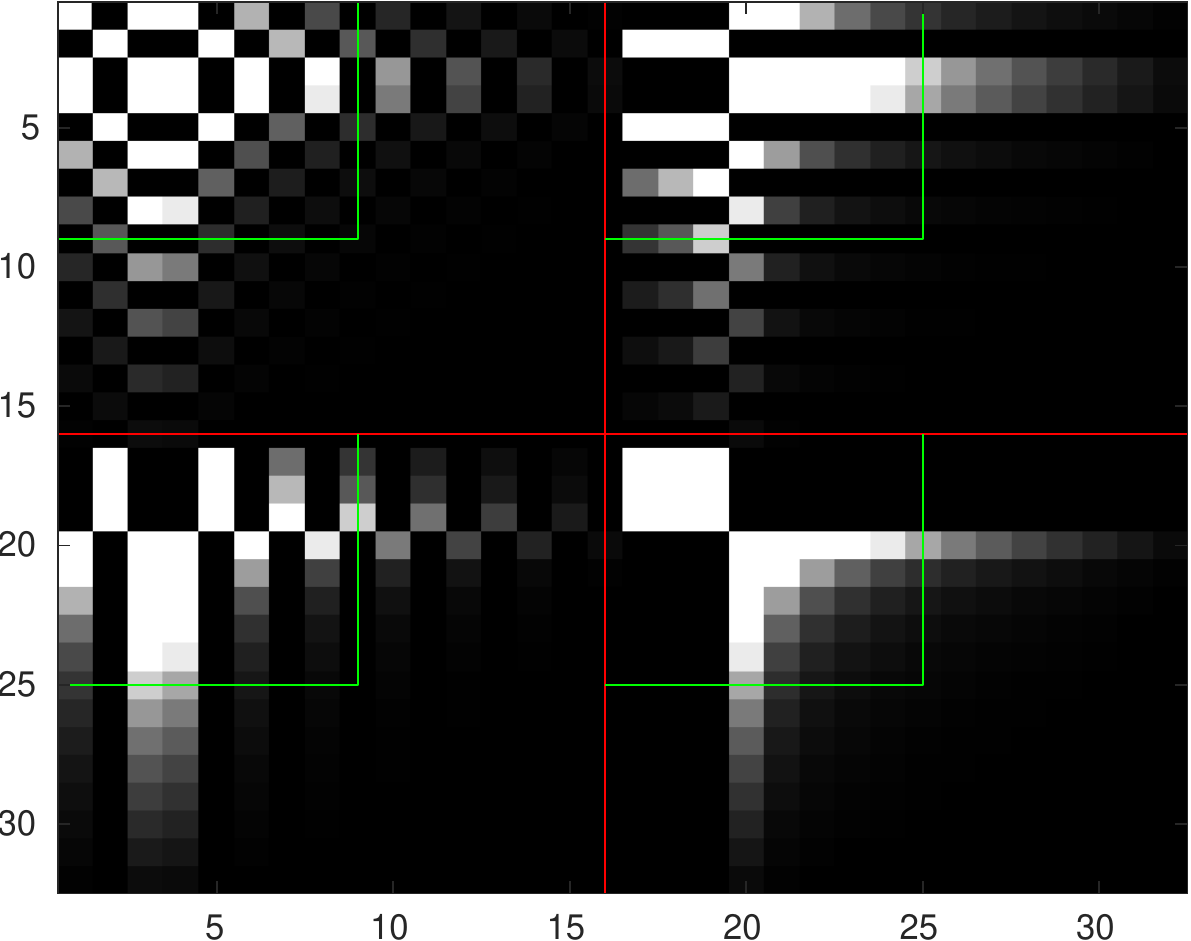}
    \caption{Bottom-left: Coefficient matrix ${\rm DCT}_R(I|_R)$ of element $R$ with size $32\times 32$ and only one non-zero coefficient $(5,5)$ with amplitude $15$. Top: Corresponding image $I|_R$ with children $R_1,\ldots, R_4$ marked in red. Bottom-right: Coefficient matrices ${\rm DCT}_{R_i}^{-1}(I|_{R_i})$. The green lines mark the block of $8\times 8$ coefficients which do not contribute to $\eta(R)$ or $\eta(R_i)$ respectively. Thus, for this example, the right-hand side of~\eqref{refprop} is zero (no non-zero coefficient outside the green block in the left image), while the left-hand side is not. }
    \label{refpropex}
\end{figure}
\begin{figure}
    \centering
    \includegraphics[width=0.49\textwidth]{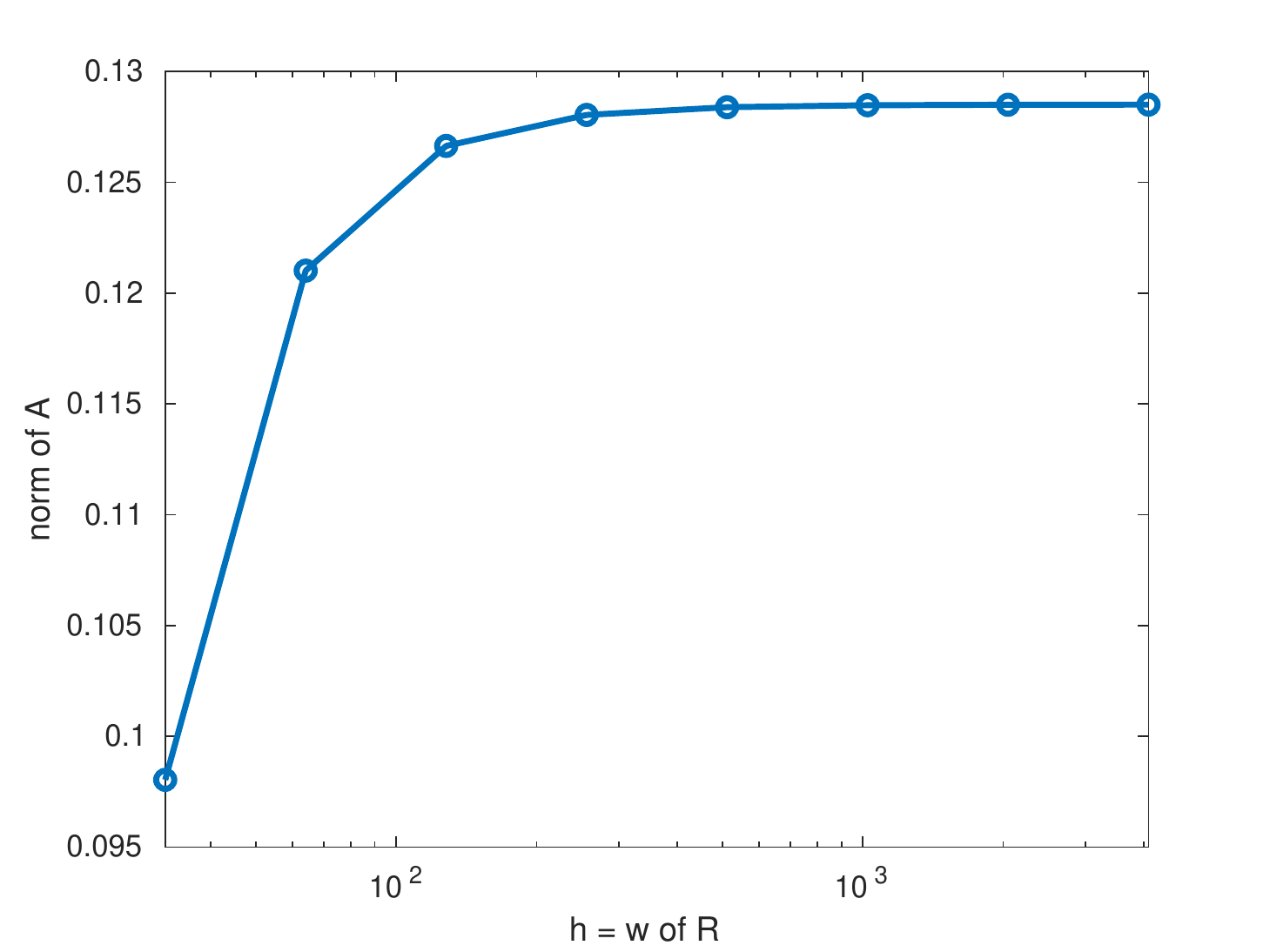}
     \includegraphics[width=0.49\textwidth]{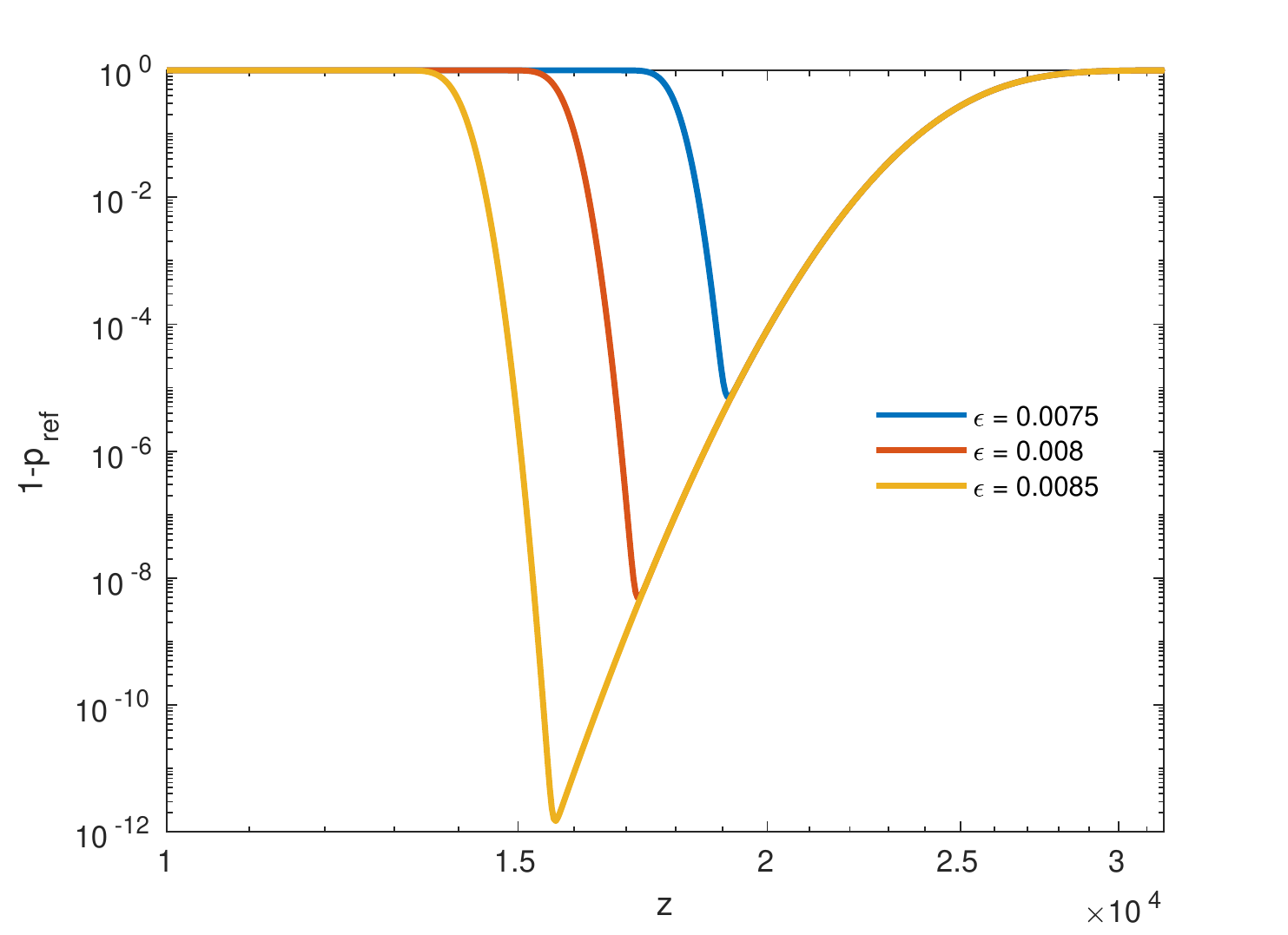}
    \caption{Left: The norm of $A$ for various sizes of $R$. Right: Upper bound on $1-p_{\rm ref}$ for different values of $\eps$ and $z\geq 0$.}
    \label{fig:prob}
\end{figure}
However, we show that under some assumptions on the distribution of images, there is a high probability that the refinement property~\eqref{refprop} holds for a given image.

\begin{lemma}\label{lem:chisquared}
Let $\Phi_{k,\lambda}$ denote the cdf of the non-central chi-squared distribution with $k\in\N$ degrees of freedom and non-centrality parameter $\lambda\geq 0$. Then, $\Phi_{k,\lambda}(x)\leq \Phi_{k,\lambda'}(x)$ for all $0\leq x<\infty$ and all $\lambda\geq \lambda'$.
\end{lemma}
\begin{proof}
 We use the well-known identity $\Phi_{k,\lambda}(x)=1-Q_{k/2}(\sqrt{\lambda},\sqrt{x})$,
 where $Q$ denotes the Marcum Q-function. According to~\cite[Equation~(2)]{esposito:1968},
 the derivative can be represented via
 \begin{align*}
     \partial_\mu Q_{k/2}(\mu,y) = \mu\big(Q_{k/2+1}(\mu,y)-Q_{k/2}(\mu,y)\big).
 \end{align*}
    Together with~\cite[Equation~(6)]{twist}, this leads to
\begin{equation}\label{lambdasign}
    \partial_\lambda \Phi_{k,\lambda}(x)=\partial_\lambda(1-Q_{k/2}(\sqrt{\lambda},\sqrt{x}))=-\frac{1}{2\sqrt{\lambda}}\left(\sqrt{\lambda}\left(\frac{\sqrt{x}}{\sqrt{\lambda}}\right)^{k/2}{\rm e}^{-\frac{\lambda+x}{2}}I_{k/2}(\sqrt{x\lambda})\right),
\end{equation}
where the modified Bessel function of the first kind is defined as
\begin{equation*}
I_{k/2}:=\sum_{k=0}^{\infty}\frac{1}{k!\Gamma (k+k/2+1)}\left(\frac{x}{2}\right)^{2k+k/2}.
\end{equation*}
This shows $\partial_\lambda \Phi_{k,\lambda}(x)\leq 0$ for all $0\leq x< \infty$ and concludes the proof.
\end{proof}

\begin{theorem}\label{thm:main}
 Let $I\in\R^{h\times w}$ denote the original image and let $\widetilde I\in\R^{h\times w}$ denote a random distortion obtained by adding $\eps \zeta_i$ to the $i$-th pixel, where the $\zeta_i$ are i.i.d. standard normal random numbers. Then, for all $C,z\geq 0$, the refinement property~\eqref{refprop} is true with  probability 
 \begin{align*}
  p_{\rm ref}\geq     \Big(1- \Phi_{448,0}\Big(\frac{\delta^2 z}{C^2-\delta^2}\Big)\Big)\Phi_{64,\eps^{-2}}(z)\quad\text{for all }C,z\geq 0
 \end{align*}
 and $C_0=2+2C$ for all $R\in\TT\in \T$.  The constant $\delta$ satisfies $\delta\leq 1$ but can numerically be tracked to $\delta\leq 0.13$ for realistic image sizes.
\end{theorem}
\begin{proof}
Let $R$ denote an element and let $\widetilde R$ denote a child of $R$. We recall the operators $\textup{TL}_{\widetilde R}$, ${\rm TL}^{-1}_{\widetilde R}$, $\textup{DCT}_{\widetilde R}$, and $\textup{IDCT}_{\widetilde R}$ defined in~\eqref{eq:ops} above.
For $i\in\{1,..,4\}$, we consider the operators
\begin{align*}
& A_{\widetilde R}\colon R\to \widetilde R\\
&A_{\widetilde R}:= (\textup{Id}-{\rm TL}^{-1}_{\widetilde R}\circ\textup{TL}_{\widetilde R})\circ \textup{DCT}_{\widetilde R}\circ\textup{P}_{\widetilde R}\circ\textup{IDCT}_R\circ{\rm TL}_R^{-1}\circ\textup{TL}_R\circ\textup{DCT}_R,
\end{align*}
where ${\rm P}_{\widetilde R}$ denotes the restriction of $R$ to $\widetilde R$ as well as
\begin{align*}
    &A\colon R\to R\\
    &A(I|_R)|_{R_i}:= A_{R_i}(I|_R),\quad\text{for }i=1,\ldots,4.
\end{align*}
As we will see below, the operator $A_{R}$ encodes the difference between the error on the coarse element $R$ and the fine element $R_i$.
As  compositions of orthogonal projections, embeddings and (inverse) discrete cosine transforms, the operators $A_{R_i}$ are bounded uniformly in $R$, $R_i$. Given $R_i$ and $R$, we may compute their norm exactly by calculating the corresponding matrix norm. Obviously, there holds $\|A\|\leq \delta=1$, but we can track the norm numerically for various realistic image sizes (see Fig.~\ref{fig:prob} for the results). This suggests
that $\delta=0.13$ is a realistic upper bound and improves the constant in the statement compared to $\delta=1$. Note that all the remaining arguments of the proof remain valid with the naive bound $\delta=1$.

We aim to argue that the refinement property~\eqref{refprop} is satisfied for most images. By assumption, we have  $\widetilde{I}(i,j)=I(i,j)+\varepsilon \xi_{i,j}$ for i.i.d. standard normal variables $\xi_{i,j}$.  Note that $\widetilde{I}$ is a random variable and the law of its $L^2$-norm follows a non-central chi-squared distribution. The local error $\eta(R_i)$ can be identified with the difference of $I_{R_i}$ and the corresponding approximate image on $R_i$. Therefore, we compute for all children $\widetilde R$ of $R$
\begin{align*}
\sum_{i=1}^4\eta(R_i)^2=& \sum_{i=1}^4\|(\textup{Id}-{\rm TL}^{-1}_{R_i}\textup{TL}_{R_i})\textup{DCT}_{R_i}{\rm P}_{R_i}(\widetilde{I})\|_2^2\\ \leq& \sum_{i=1}^42\|(\textup{Id}-{\rm TL}^{-1}_{R_i}\textup{TL}_{R_i})\textup{DCT}_{R_i}{\rm P}_{R_i}\textup{IDCT}_R(\textup{Id}-{\rm TL}^{-1}_{R}\textup{TL}_{R})\textup{DCT}_R(\widetilde{I})\|_2^2\\&+2\|A(\widetilde{I})\|_2^2\\
\leq& \sum_{i=1}^4 2\|{\rm P}_{R_i}\textup{IDCT}_R(\textup{Id}-{\rm TL}^{-1}_{R}\textup{TL}_{R})\textup{DCT}_R(\widetilde{I})\|_2^2\\&+2\|A(\widetilde{I})\|_2^2\\
\leq &2\|(\textup{Id}-{\rm TL}^{-1}_{R}\textup{TL}_{R})\textup{DCT}_R(\widetilde{I})\|_2^2+2\delta\|\widetilde{I}\|_2^2
=2\eta(R)^2 +2\delta\|\widetilde{I}\|_2^2,
\end{align*}
where we used the triangle inequality and the identity
\begin{equation*}
\textup{Id}=\textup{IDCT}_R(\textup{Id}-{\rm TL}^{-1}_{R}\textup{TL}_{R})\textup{DCT}_R+\textup{IDCT}_R({\rm TL}^{-1}_{R}\textup{TL}_{R}(\textup{DCT}_R)).
\end{equation*}
Next, we show for $C>0$ that
\begin{equation}\label{namelessequ}
\delta/C\|\widetilde{I}\|_2=\delta/C\|{\rm DCT}_R(\widetilde{I})\|_2\leq \|(\textup{Id}-{\rm TL}^{-1}_{R}\textup{TL}_{R})\textup{DCT}_R(\widetilde{I})\|_2=\eta(R)
\end{equation}
holds with high probability. We denote this probability with $0\leq p_{\rm ref}\leq 1$.
Note that adding the noise to $I$ is equivalent to adding the noise to $\textup{DCT}_R(\widetilde{I})$, since the discrete cosine transform is just a change of orthonormal bases.
Thus, it remains to estimate the norm of
\begin{align*}
 \Big((\textup{Id}-{\rm TL}^{-1}_{R}\textup{TL}_{R}) \widetilde J\Big)_{ij}=\begin{cases}
 0 & i\leq 8 \text{ or }j\leq 8,\\
 \widetilde J_{ij}&\text{else},
 \end{cases}
\end{align*}
where $\widetilde J := ({\rm DCT}_{R}(I)_{i,j} + \eps \xi_{i,j})$. Thus, it suffices to estimate the probability for 
\begin{align*}
    \delta/C\norm{\widetilde J}{2}\leq \norm{(\textup{Id}-{\rm TL}^{-1}_{R}\textup{TL}_{R}) \widetilde J}{2}.
\end{align*}
With $A:=\set{(i,j)\in\N^2}{1\leq i\leq h_R,\,1\leq j\leq w_R}$ and $B:= A\setminus \{1,\ldots,8\}^2$, this can be rewritten as
\begin{align*}
   p_{\rm ref}= \P\Big( \delta^2/C^2&\sum_{(i,j)\in A} ({\rm DCT}_R( I)_{i,j} + \eps\xi_{i,j})^2\leq 
     \sum_{(i,j)\in B} ({\rm DCT}_R( I)_{i,j} + \eps\xi_{i,j})^2\Big).
     % &=
     %  \P\Big( \delta^2/C^2\sum_{1\leq i,j\leq 8} ({\rm DCT}_R( I)_{i,j} + \eps\xi_i)^2\leq 
     % (1-\delta^2/C^2)\sum_{(i,j)\in B} ({\rm DCT}_R( I)_{i,j} + \eps\xi_i)^2\Big)\\
     % &\geq 
     %  \P\Big( 2\delta^2/C^2\sum_{1\leq i,j\leq 8} {\rm DCT}_R( I)_{i,j}^2 + \eps^2\xi_i^2\leq 
     % (1-\delta^2/C^2)\sum_{(i,j)\in B} ({\rm DCT}_R( I)_{i,j} + \eps\xi_i)^2\Big).
\end{align*}
Without loss of generality, we may assume that $\norm{I}{2}=1$ and hence $\sum_{1\leq i,j\leq 8} {\rm DCT}_R( I)_{i,j}^2\leq 1$, which leads to
\begin{align*}
    p_{\rm ref}\geq \P\Big( \delta^2/C^2 X\leq  (1-\delta^2/C^2)Y\Big),
\end{align*}
with independent random variables $X:=\sum_{1\leq i,j\leq 8} ({\rm DCT}_R( I)_{i,j}/\eps +\xi_{i,j})^2$ and $Y:=\sum_{(i,j)\in B} (\eps^{-1}{\rm DCT}_R( I)_{i,j} + \xi_{i,j})^2$.  We set $\mu=\delta^2/C^2$ and denote by $\phi_X$, $\phi_Y$ and $\Phi_X$, $\Phi_Y$ the respective density and cumulative density functions of $X$ and $Y$.  This leads to
\begin{align*}
    p_{\rm ref}&= \int_0^\infty \int_{\mu x/(1-\mu)}^\infty \phi_Y(y)\,dy \phi_X(x)\,dx=
    \int_0^\infty \Big(1-\Phi_Y(\frac{\mu x}{1-\mu})\Big)\phi_X(x)\,dx\\
    &\geq 
     \int_0^\infty \Big(1-\Phi_{448,0}(\frac{\mu x}{1-\mu})\Big)\phi_X(x)\,dx,
\end{align*}
where we used Lemma~\ref{lem:chisquared} for the last inequality as well as the fact that $Y$ contains at least $32\cdot 16-64=448$ terms. This follows from the fact that Algorithm~\ref{alg:adaptive} only refines elements with $h_R,w_R\geq 16$. For $h_R=w_R=16$, the left-hand side of~\ref{refprop} is zero by definition. Since $w_R$, $h_R$ are powers of two by assumption, the next smallest element $R$ satisfies $\max\{h_R,w_R\}\geq 32$, such that $\#B\geq 32\cdot 16-64 = 448$.
For $z\geq0$, we may estimate
\begin{align*}
     p_{\rm ref}&\geq 1-
    \int_0^\infty \Phi_{448,0}(\frac{\mu x}{1-\mu})\phi_{64,0}(x)\,dx\geq 1- 
     \int_0^z \Phi_{448,0}(\frac{\mu x}{1-\mu})\phi_{X}(x)\,dx - (1-\Phi_{64,\eps^{-2}}(z))\\
     &\geq 
     (1- \Phi_{448,0}\Big(\frac{\mu z}{1-\mu}\Big))\Phi_{64,\eps^{-2}}(z).
\end{align*}
Combining the estimates, we obtain with probability $p_{\rm ref}$ that
\begin{align*}
\sum_{i=1}^4 \eta(\widetilde R_i)^2\leq (2+2C)\eta (R)^2.
\end{align*}
This concludes the proof.
\end{proof}

The cdfs of non-centralized chi squared distributions can be found, e.g. in \cite{chi:2022} and can be numerically calculated in, e.g., Matlab via the command \texttt{ncx2cdf}.
This allows us to evaluate the probability bound from Theorem~\ref{thm:main} for realistic values of $\delta=0.13$ and $C=1$. This implies the refinement property~\eqref{refprop} with $C_0=4$ with probability
\begin{equation*}
p_{\rm ref}\geq
\left\{
\begin{array}{ll}
1- 7.3\cdot 10^{-6}& \varepsilon=0.0075, \\
1-5.4\cdot 10^{-9} & \varepsilon=0.008, \\
1-1.6\cdot 10^{-12} & \eps = 0.0085.
\end{array}
\right.
\end{equation*}
See Figure~\ref{fig:prob} for the behavior of the lower bound for $p_{\rm ref}$ from Theorem~\ref{thm:main} with respect to $z$.

\begin{remark}\label{rem:probexample}
We note that Theorem~\ref{thm:main} implies that a noise amplitude $\varepsilon = 0.0075$, ensures a probability of 99.99\% to satisfy the refinement property~\eqref{refprop} with $C_0=4$. We also emphasize that such a noise is not noticeable to the human eye as demonstrated in Figure~\ref{noise} even for an amplitude of 0.015. This suggests that images which do not satisfy~\eqref{refprop} are the rare exception.
\end{remark}
\begin{figure}[h]
\centering
	\includegraphics[scale=0.05]{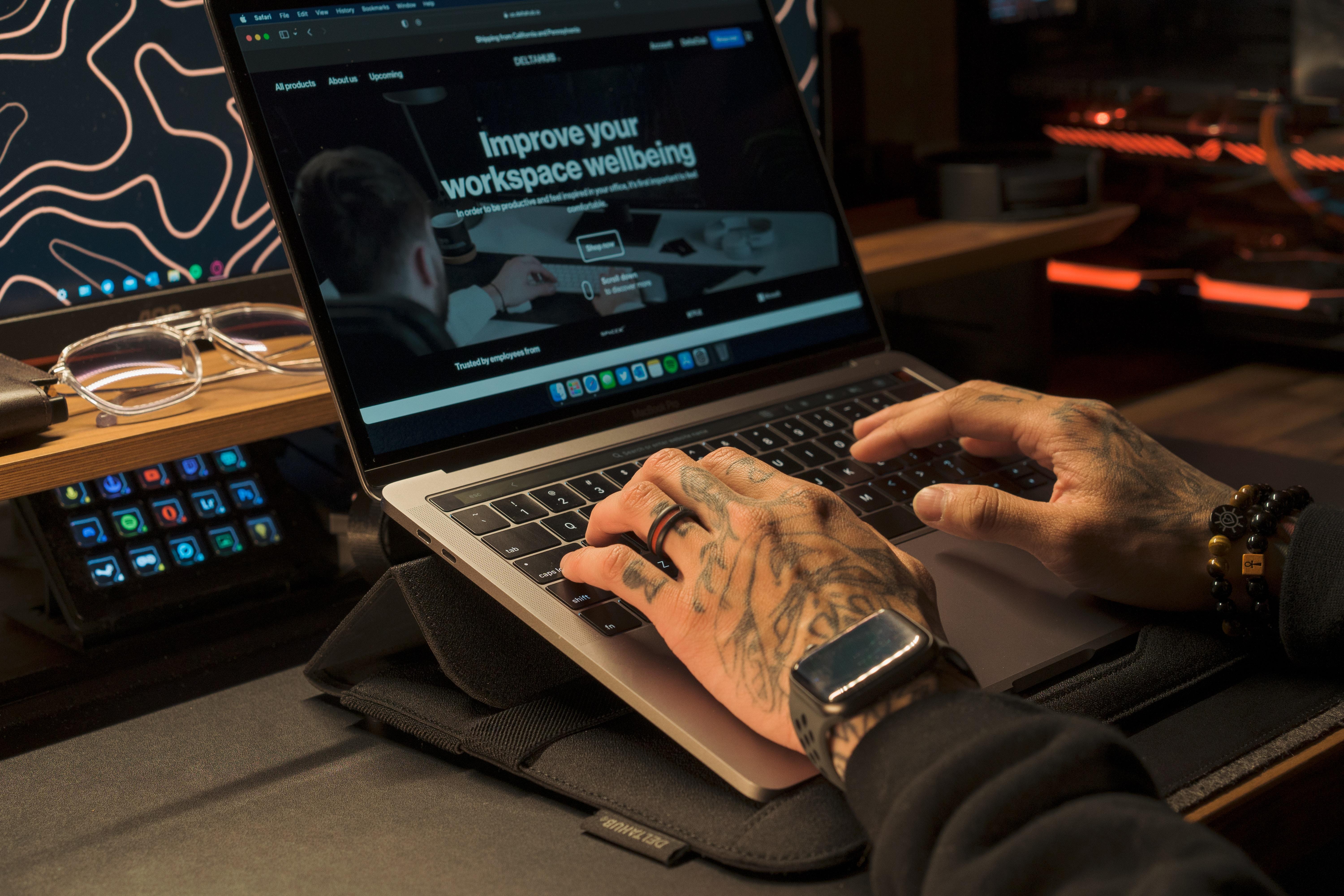}\\
    \includegraphics[scale=0.05]{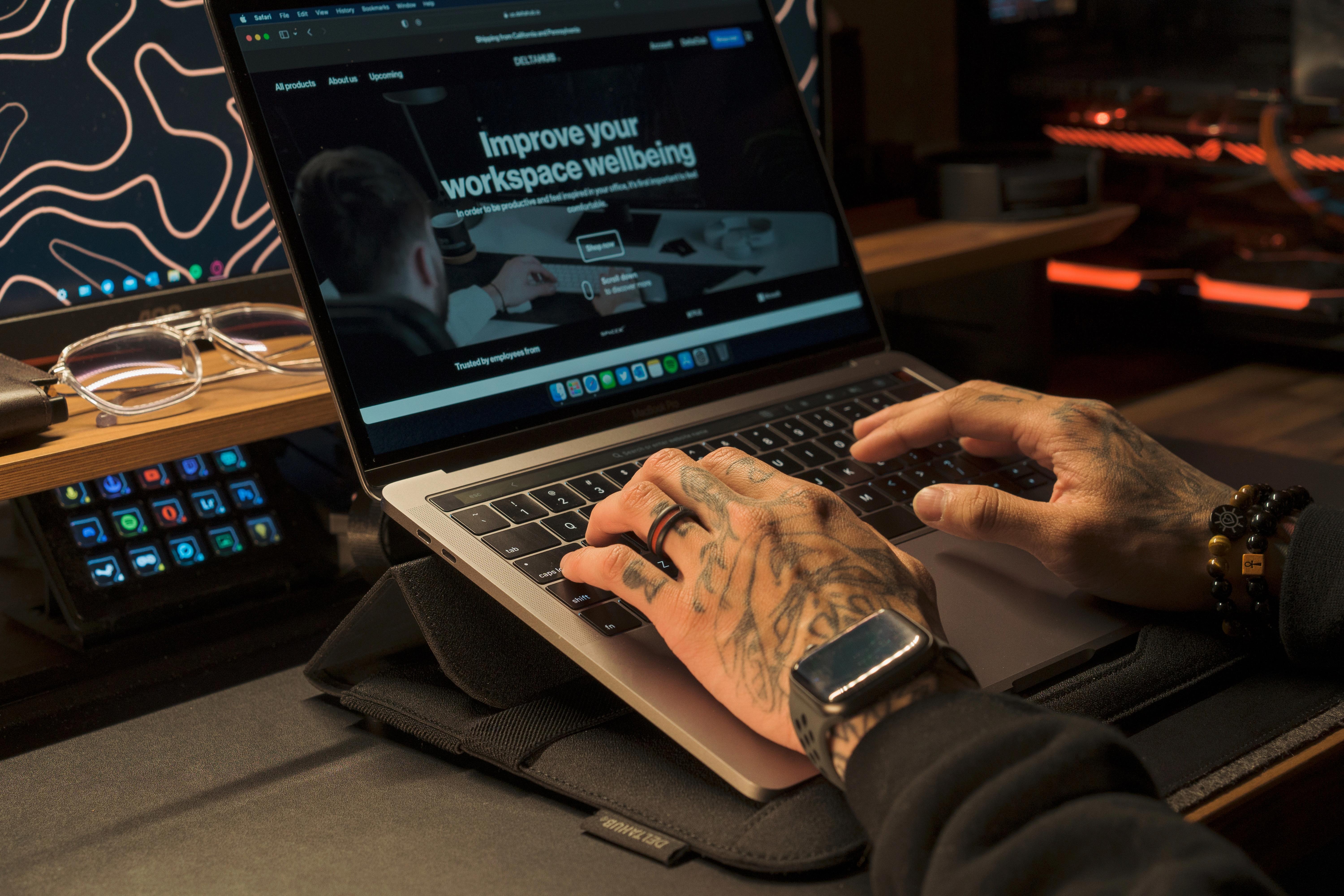}
	\caption{Top: The original image. Bottom: The original image with an additional noise of amplitude 0.015. Image from unsplash.com by artist Kadyn Pierce.}
	\label{noise}
\end{figure}

The following result states the optimality of Algorithm~\ref{alg:adaptive}.
\begin{theorem}\label{thm:optimality}
Let $I$ be such that the refinement property~\eqref{refprop} is true.
Then, for all iterations $\ell=0,~1,\ldots$ of Algorithm~\ref{alg:adaptive}, the mesh $\TT_\ell$ satisfies
	\begin{equation}\label{5p8}
	E(\TT_\ell)\leq 12(2+2C)^2E_m\quad\text{for all }m\leq \#\TT_\ell/10,
	\end{equation}
 where $C>0$ is the constant from Theorem~\ref{thm:main}.
	 To create $\TT$ the algorithm uses at most $12(2+2C)^2(\#\TT_\ell-\#\TT_0+1)$ arithmetic operations and computations of $\eta$.
\end{theorem}
\begin{proof}
	Since the refinement property~\eqref{refprop} is true, the statement follows from~\cite[Theorem~5.2]{binev:2004}, where the number of children is $K=4$. We note that we use a slightly different definition of $\T_n$ and hence $E_n$. In~\cite{binev:2004}, $\T_n$ denotes the set of meshes generated by exactly $n$ refinements, while we count the number of elements. Since each refinement generates exactly three new elements, the number of refinements corresponds to $(\#\TT-\#\TT_0)/3$. Thus, both definitions are equivalent and the statement above is proved.
\end{proof}
\begin{remark}
    Using the values $\eps=0.0075$ and $C=1$ from Remark~\ref{rem:probexample}, Theorems~\ref{thm:main}--\ref{thm:optimality} imply $E(\TT_\ell)\leq 192 E_m$ for all $m\leq \#\TT_\ell/10$ with probability over $99.99\%$.
\end{remark}
\begin{remark}
In Algorithm \ref{alg:adaptive} the use of the largest $\widetilde{\eta}(R)$ may require a sorting procedure with complexity $O(N\textup{log}N)$. To keep the number of operations of order $O(N)$ we can use binary bins instead of sorting. The result and proof can be changed accordingly with only the absolute constant in the estimate of the theorem changing.
\end{remark}

\subsection{Adaptive compression in other norms}\label{sec:bvnorm}
A norm which might be better suited to capture the quality loss in images is the $BV$-norm. A heuristic explanation for this is the fact that the human eye is more sensitive to sharp color/brightness changes which are not captured in the $L^2$-norm.
The $BV$-norm for a function $u\in BV\subset L^1((0,h]\times (0,w])$ is defined by
\begin{equation}\label{bvdef}
\|u\|_{BV}:= \|u\|_1+V(u),
\end{equation}
where
\begin{equation*}
V(u):= \textup{sup}\left\{\int_{(0,h]\times (0,w]}u(x)\textup{div}(\phi (x))dx:\phi \in C_c^1((0,h]\times (0,w],\mathbb{R}^n),~\|\phi\|_{\infty}\leq 1\right\}
\end{equation*}
denotes the total variation of $u$.
For weakly differentiable functions $w\in L^1$, i.e., $w\in W^{2,1}$ the $BV$ norm of $w$ is given by 
\begin{equation*}
\|w\|_{BV}=\|w\|_1+\|\nabla w\|_1.
\end{equation*}
However, in the present setting, the images $I$ and $I_{{\rm app}}$  represent functions which are constant on each pixel~$\square$ of the image. For such a function $u$, the total variation computes to
\begin{align*}
    V(u) = \sup_{\norm{\phi}{\infty}\leq 1} \sum_{\square \in I} \int_\square u{\rm div}(\phi)\,dx = 
    \sup_{\norm{\phi}{\infty}\leq 1} \sum_{\square \in I} \int_{\partial\square\cap I^\circ} u n\cdot \phi\,dx = \sum_{e\in F^\circ} \norm{[u]}{L^1(e)},
\end{align*}
where $n$ denotes the unit normal vector on $\partial\square$, $I^\circ$ denotes the open subset of $\R^2$ occupied by the pixels of $I$, and $F^\circ$ denotes the set of interior interfaces between the pixels. We also use the common notation $[u]|_e= u|_{\square}-u|_{\square'}$ to denote the jump of $u$ over the interface $e=\square\cap \square'\in F^\circ$ between two pixels $\square$ and $\square'$.
Thus the $BV$ norm of the approximation error $(I-I_{{\rm app}})$ can be written as
\begin{equation*}
\|I-I_{{\rm app}}\|_{BV}=\frac{1}{hw}\Big(\|I-I_{{\rm app}}\|_1+ \sum_{e\in F^\circ}\big|[I-I_{{\rm app}}]|_e\big|\Big).
\end{equation*}
The new norm induces new local and global errors
\begin{align}\label{eq:BVest}
\eta(R)=\|(I-I_{{\rm app}})|_R\|_{BV} \quad\text{and}\quad 
E(\TT)=\sum_{R\in\TT}\eta(R)
\end{align}
which can be used to drive Algorithm~\ref{alg:adaptive}.
\begin{remark}
Since we consider norms on finite-dimensional spaces, we immediately obtain equivalence of the $L^2$-norm and $BV$-norm. Therefore, the refinement property~\eqref{refprop} also holds for the $BV$-norm for most images and hence also the near-optimality result Theorem~\ref{thm:optimality}. However, the constants in~\eqref{refprop} for the $BV$-norm depend on the norm equivalence constants, which in turn depend on the image size. A direct computation of the constant seems very difficult as it would involve mapping properties of the discrete cosine transform as an operator on $BV$.
\end{remark}	
	\section{Efficient storage of the adaptive compression}\label{sec:storage}
	The adaptive grid produced by Algorithm~\ref{alg:adaptive} has to be stored in order to reconstruct the compressed image. This potentially produces overhead compared to standard JPEG compression, which uses a predefined grid. However, using some simple optimizations, we demonstrate that this overhead is not significant.
	
	\subsection{Storing $I_{{\rm app}}$}\label{sec:storing}
	Algorithm~\ref{alg:adaptive} is applied to each of the color components $X\in \{Y,C_b,C_r\}$ and produces a corresponding mesh $\TT^X$.
 As we neglected the rounding step of the JPEG compression in the adaptive algorithm, we compress the $X$ component of $I_{\rm app}$, $I^X_{app}$, for each element $R\in\TT^X$ by:
	\begin{itemize}
		\item Compute $F^X(R):= \textup{round}(\textup{TL}_R(\textup{DCT}_R(I^X|_R))./Q)$ 
		\item Store entries from $F^X(R)$ in the order depicted in Figure~\ref{fig:enum} in a vector $S_R^X$ until the last non-zero entry 
	\end{itemize}	
 Note that the discrete cosine transform returns a matrix of coefficients of basis functions, where the coefficients towards bottom and right correspond to basis functions of higher frequency. Typical images usually have small coefficients corresponding to basis functions of high frequency, which means that after rounding they become zero and are dropped.

	\begin{figure}[btp]
		\centering
		\begin{tikzpicture}[scale=0.45]
		
		\draw (0.75,11.25) node{1};
		\draw (2.25,11.25) node{3};
		\draw (3.75,11.25) node{6};
		\draw (5.25,11.25) node{10};
		\draw (6.75,11.25) node{15};
		\draw (8.25,11.25) node{21};
		\draw (9.75,11.25) node{28};
		\draw (11.25,11.25) node{36};
		\draw (0.75,9.75) node{2};
		\draw (2.25,9.75) node{5};
		\draw (3.75,9.75) node{9};
		\draw (5.25,9.75) node{14};
		\draw (6.75,9.75) node{20};
		\draw (8.25,9.75) node{27};
		\draw (9.75,9.75) node{35};
		\draw (11.25,9.75) node{43};
		\draw (0.75,8.25) node{4};
		\draw (2.25,8.25) node{8};
		\draw (3.75,8.25) node{13};
		\draw (5.25,8.25) node{19};
		\draw (6.75,8.25) node{26};
		\draw (8.25,8.25) node{34};
		\draw (9.75,8.25) node{42};
		\draw (11.25,8.25) node{49};
		\draw (0.75,6.75) node{7};
		\draw (2.25,6.75) node{12};
		\draw (3.75,6.75) node{18};
		\draw (5.25,6.75) node{25};
		\draw (6.75,6.75) node{33};
		\draw (8.25,6.75) node{41};
		\draw (9.75,6.75) node{48};
		\draw (11.25,6.75) node{54};
		\draw (0.75,5.25) node{11};
		\draw (2.25,5.25) node{17};
		\draw (3.75,5.25) node{24};
		\draw (5.25,5.25) node{32};
		\draw (6.75,5.25) node{40};
		\draw (8.25,5.25) node{47};
		\draw (9.75,5.25) node{53};
		\draw (11.25,5.25) node{58};
		\draw (0.75,3.75) node{16};
		\draw (2.25,3.75) node{23};
		\draw (3.75,3.75) node{31};
		\draw (5.25,3.75) node{39};
		\draw (6.75,3.75) node{46};
		\draw (8.25,3.75) node{52};
		\draw (9.75,3.75) node{57};
		\draw (11.25,3.75) node{61};
		\draw (0.75,2.25) node{22};
		\draw (2.25,2.25) node{30};
		\draw (3.75,2.25) node{38};
		\draw (5.25,2.25) node{45};
		\draw (6.75,2.25) node{51};
		\draw (8.25,2.25) node{56};
		\draw (9.75,2.25) node{60};
		\draw (11.25,2.25) node{63};
		\draw (0.75,0.75) node{29};
		\draw (2.25,0.75) node{37};
		\draw (3.75,0.75) node{44};
		\draw (5.25,0.75) node{50};
		\draw (6.75,0.75) node{55};
		\draw (8.25,0.75) node{59};
		\draw (9.75,0.75) node{62};
		\draw (11.25,0.75) node{64};
		\draw[step=1.5cm,black,very thin] (0,0) grid (12,12);
		\end{tikzpicture}
		\caption{Visualizing the enumeration $N$ of the matrix entries}\label{fig:enum}
	\end{figure}	
To reconstruct $I_{{\rm app}}^X$ from the vectors $S_R^X$, we also need to store the width $w$ of $I_{{\rm app}}$, the ratio $h/w$ of $I_{{\rm app}}$ as well as the width of $R\in\TT^X$. Furthermore, we need some separator bits denoted by \emph{sep} in order to mark the start of the next element in $S_I$. 
Surprisingly, it is not necessary to store the location of the element $R$ within $I$.
To that end, we introduce a total order on the elements $R$ within a mesh $\TT^X$. Recall that an element $R$ is given by the position of its top-left pixel and its width. Therefore, we can simply order the top-left pixels of the elements:
	
We define the lexicographic order $<$ on the pixels $(i,j)$ of $I$ by
	\begin{equation}\label{eq:order}
	(i,j)<(i',j')\quad\Longleftrightarrow \quad (i<i')\text{ or } \big( (i=i')\text{ and }(j<j')\big).
	\end{equation}
	We then sort the elements $\TT^X=\{R_1^X,\ldots R_{n_X}^X\}$ such that their top-left pixels $(j_{R_i},k_{R_i})$ are ordered with respect to $<$ in an ascending fashion.
	
	We consider the vector $S_I$ for an approximate image $I_{{\rm app}}$ with corresponding meshes $\TT^X$ consisting of elements $R_i^X$, $i=1,..,n_X$ with $X\in \{Y,C_b,C_r\}$. Then the vector $S_I$ has the form
	\begin{align*}
	S_I=\Big(&
	\frac{h}{w},\, 	w ,\,	\textup{width}(R_1^Y) ,\,	S_{R_1^Y}^Y,\,
	{\rm sep},\,
	\textup{width}(R_2^Y) ,\,	S_{R_2^Y}^Y,\,
	{\rm sep},\, \ldots,\,
    \textup{width}(R_{n_Y}^{Y}) ,\,	S_{R_{n_Y}^{Y}}^{Y},\\
 & \qquad \textup{width}(R_1^{C_b}) ,\,	S_{R_1^{C_b}}^{C_b},\,
	{\rm sep},\,\ldots,\, \textup{width}(R_{n_{C_b}}^{C_b}) ,\,	S_{R_{n_{C_b}}^{C_b}}^{C_b},\,\\
 & \qquad \textup{width}(R_1^{C_r}) ,\,	S_{R_1^{C_r}}^{C_r},\,
	{\rm sep},\,\ldots,\, \textup{width}(R_{n_{C_r}}^{C_r}) ,\,	S_{R_{n_{C_r}}^{C_r}}^{C_r}
	\Big).
	\end{align*}
	A further lossless runlength compression step~\cite{hopkins:2021}  is applied to $S_I$ before storing the data. This is similar to the standard JPEG compression and removes the typically many zeros appearing in $S_R^X$.
	
	\subsection{Reconstruction of the compressed image}\label{sec:reconstruct}
	To reconstruct $I_{\rm app}$, we first reverse the lossless compression of $S_I$. 
 We remove $h/w$, and $w$ from $S_I$ and store them.
 Then, we reconstruct each color component $X\in\{Y,C_b,C_r\}$ individually.
 The reconstruction works as follows:
 \begin{algorithm}\label{alg:reconstruct}
 \begin{enumerate}
    \item Initialize empty $I_{\rm app}^X$ with size $h\times w$.
    \item Find the $<$-minimal empty pixel $(i,j)$ in $I_{\rm app}^X$, if no empty pixel exists move to the next color component.
     \item Remove the first available block $(\textup{width}(R^X) ,\,	S_{R^X}^X)$ from $S_I$.
     \item Compute $w_{R^X}=\textup{width}(R^X)$ and $h_{R^X}=w_{R^X}\frac{h}{w}$.
     \item Set 
     \begin{align*}
I_{\rm app}^X|_{\{i,\ldots,i+h_{R^X}-1\}\times \{j,\ldots,j+w_{R^X}-1\}}=\textup{IDCT}_R({\rm TL}^{-1}_{R^X}(Q\odot F^X(R^X))).
	\end{align*}
     and goto Step~2.
\end{enumerate}
\end{algorithm}
\begin{remark}
Searching for the minimal pixel in Step~2 of Algorithm~\ref{alg:reconstruct} is quite costly if done in a naive way. However, we notice that the minimum is always located next to the corner pixels of previously reconstructed elements $R$ (precisely, next to the top-right and bottom-left pixels, see Figure~\ref{fig:list}). Updating a list of those candidate pixels reduces the search effort significantly.
\end{remark}
\begin{figure}
		\centering
		\begin{tikzpicture}[scale=0.4]
		\draw (0,0) rectangle (16,16);
		\draw (0,12) rectangle (4,16);
		\draw (4,12) rectangle (8,16);
		\draw (8,8) rectangle (16,16);
		\draw (0,10) rectangle (2,12);
		\fill[red] (0,9.75) rectangle (0.25,10);
		\draw (0,9.75) rectangle (0.25,10);
		\fill[red] (2,11.75) rectangle (2.25,12);
		\draw (2,11.75) rectangle (2.25,12);
		%\fill[red] (4,11.75) rectangle (4.25,12);
		%\draw (4,11.75) rectangle (4.25,12);
		\fill[red] (8,7.75) rectangle (8.25,8);
		\draw (8,7.75) rectangle (8.25,8);
		\draw (2.125,12.1)--(2.125,14.1);
		\draw (1.8,12.6)--(2.125,12.1);
		\draw (2.45,12.6)--(2.125,12.1);
		\end{tikzpicture}
		\caption{Visualizing the list of candidates for the $<$-minimal empty pixel. The arrow marks the minimum.}\label{fig:list}
\end{figure}
	
 \section{Practical performance of the adaptive compression algorithm}\label{sec:experiments}
We perform a limited number of experiments with Algorithm~\ref{alg:adaptive} in order to test its performance on realistic images depicted in Figure~\ref{fig:meine-grafik}. As shown in Figures~\ref{fig:qual}--\ref{fig:qual2}, we choose the tolerance $\tau$ in Algorithm~\ref{alg:adaptive} such that no quality difference between original and compression is immediately visible. Figure~\ref{fig:storage} compares the storage size of the original JPEG file and the adaptive recompression with Algorithm~\ref{alg:adaptive}.

We observe that storage space for larger JPEG Files can be significantly reduced with adaptive compression, while smaller files show little improvement. 
Since Algorithm~\ref{alg:adaptive} does not produce elements smaller than $8\times 8$, the finest obtainable mesh is the one of the standard JPEG algorithm. Hence, heuristically, the required storage space for any given image is at most about the same as JPEG, although a rigorous statement on the worst case is non-trivial.	
 
\begin{figure}
 \centering
            \includegraphics[height=34mm]{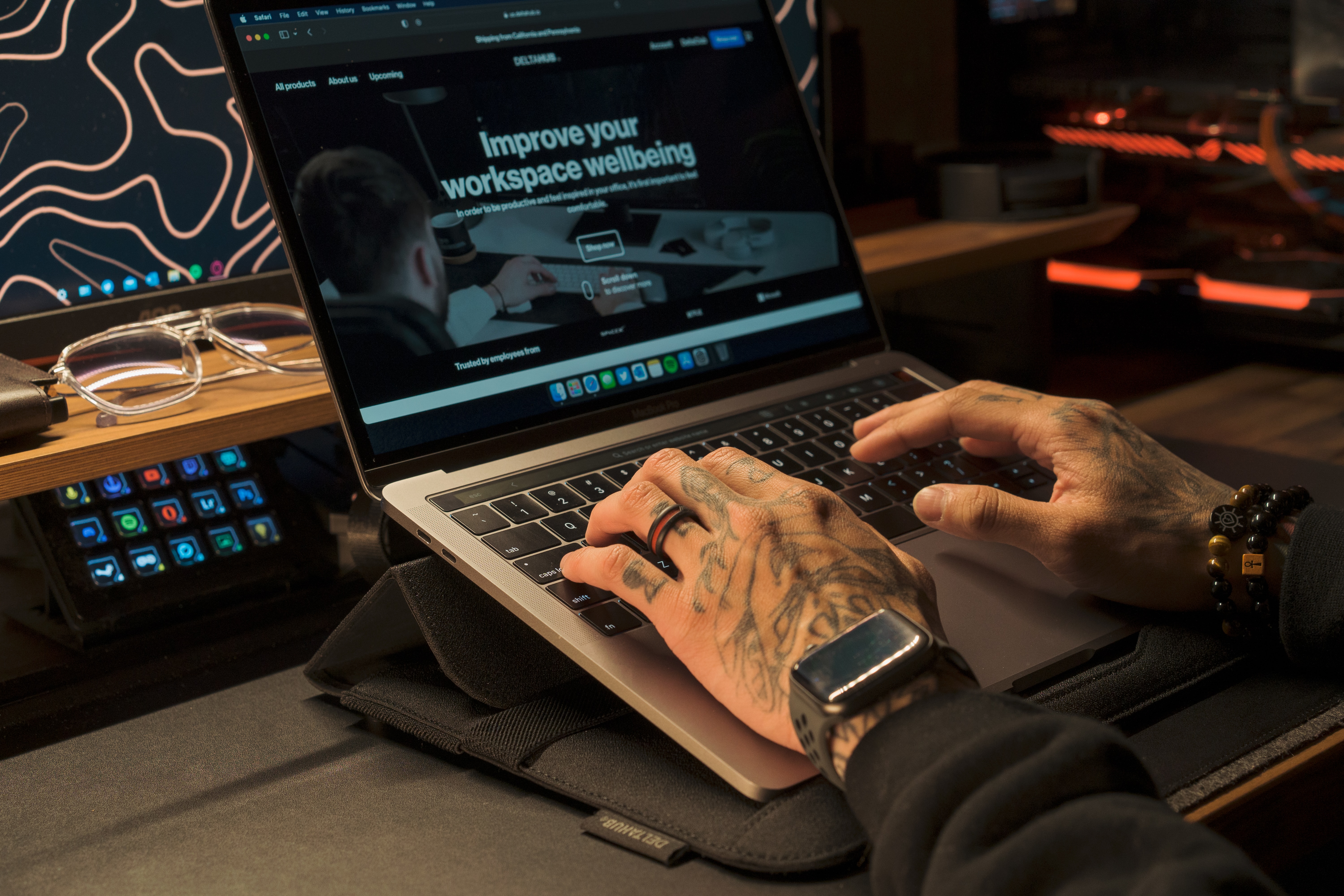}
		\includegraphics[height=34mm]{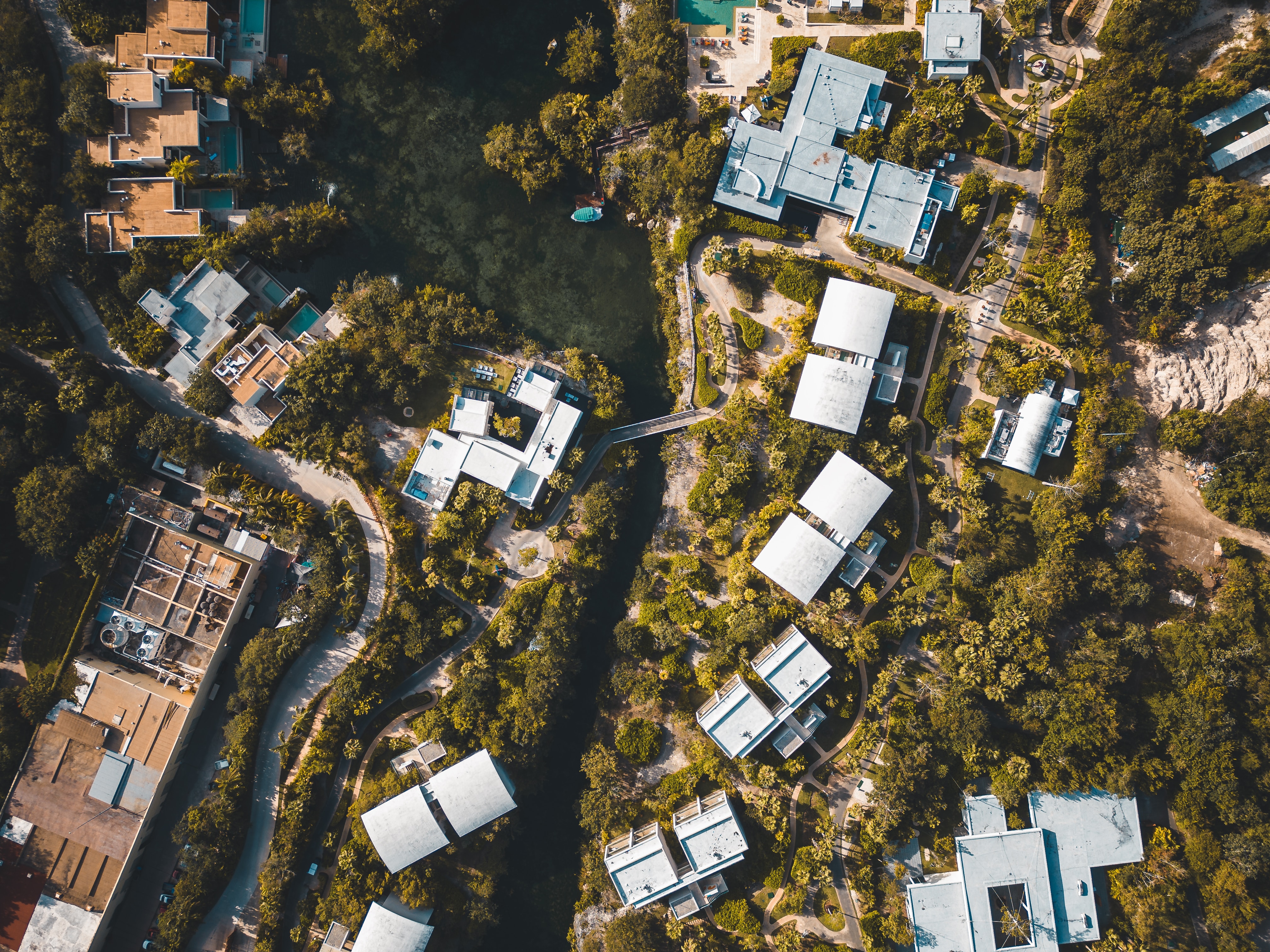}
		\includegraphics[height=34mm]{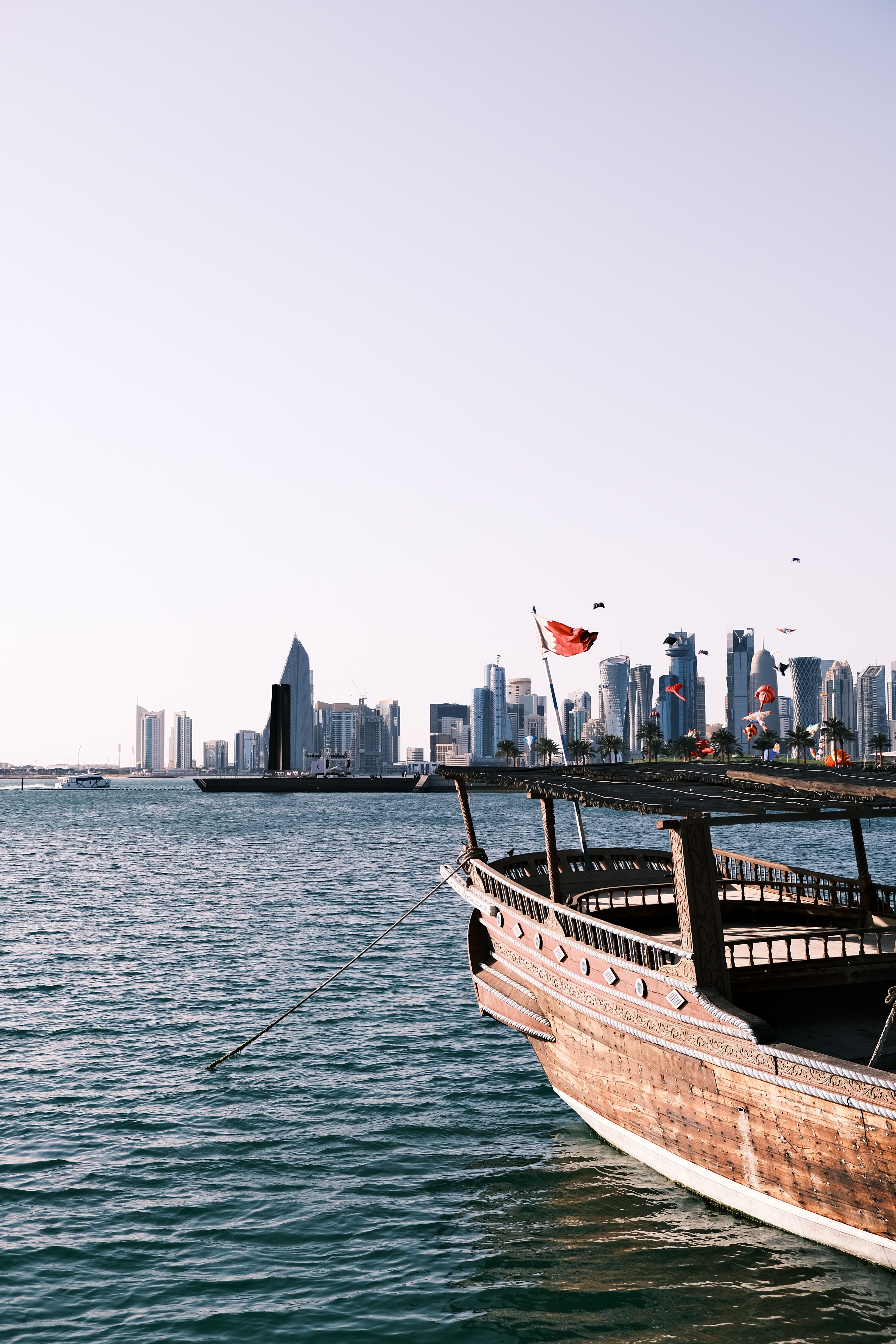}
		\includegraphics[height=34mm]{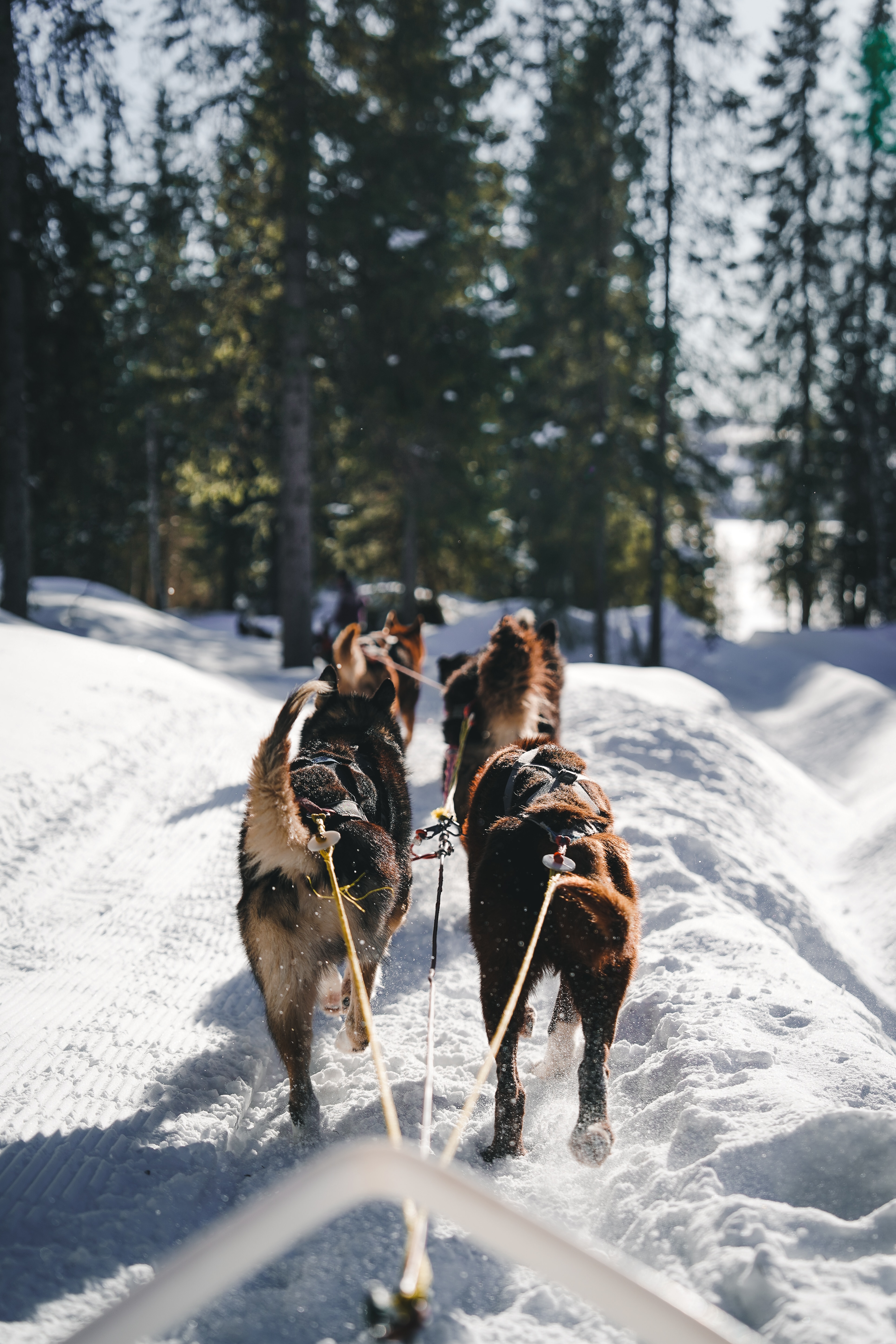}
		\hfill\\
		\includegraphics[height=34mm]{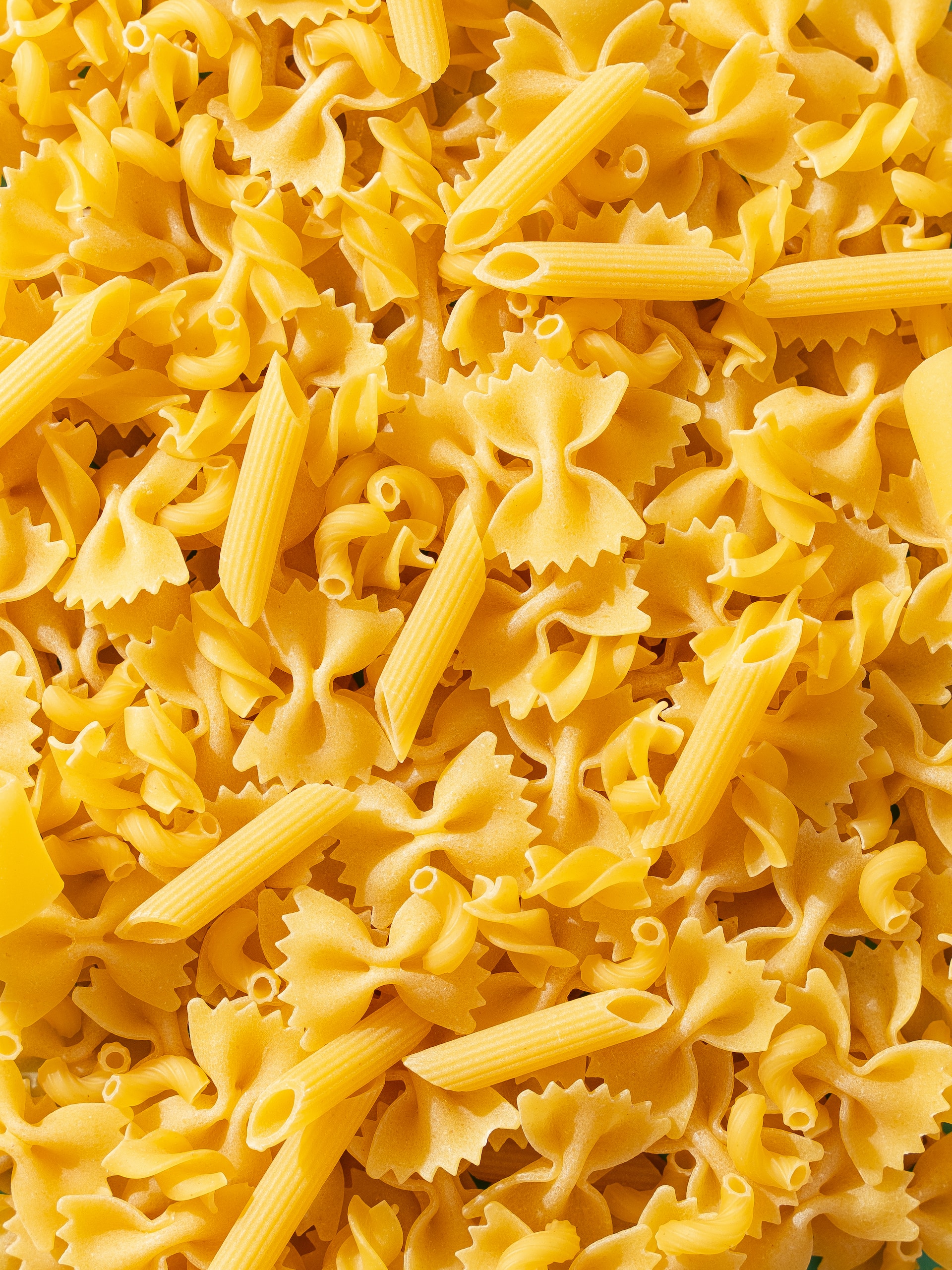}
		\includegraphics[height=34mm]{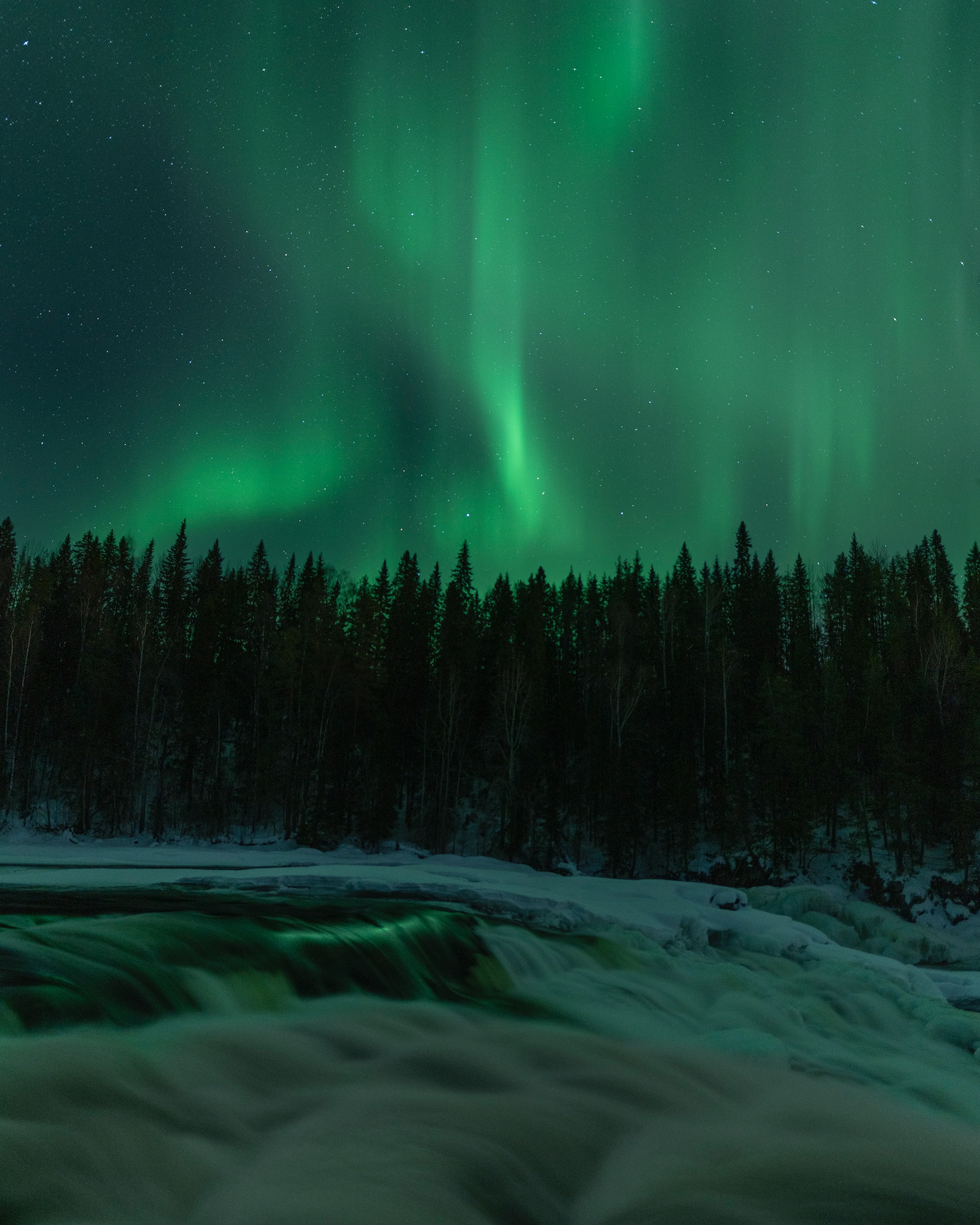}
		\includegraphics[height=34mm]{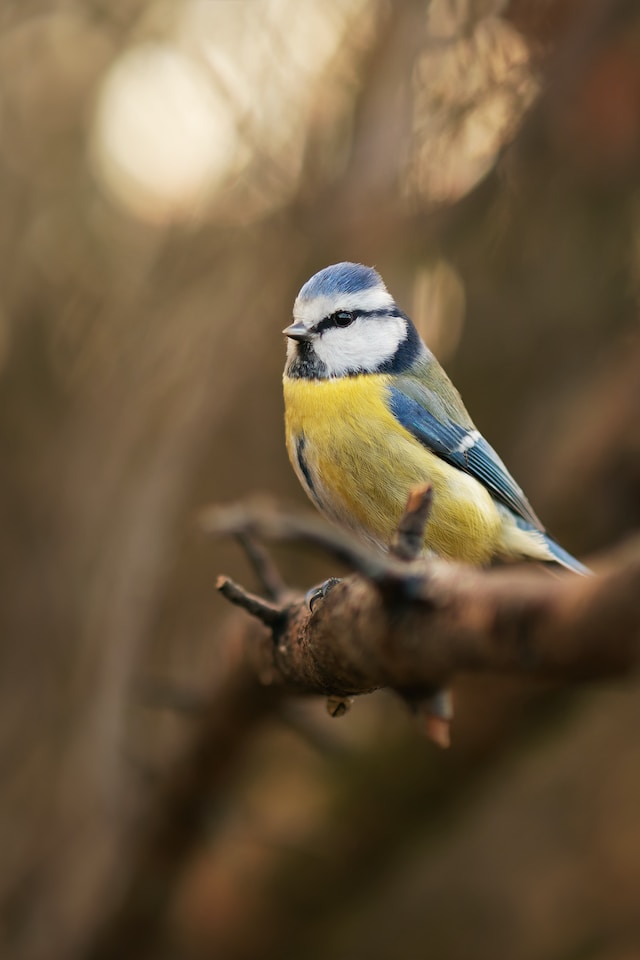}
		\includegraphics[height=34mm]{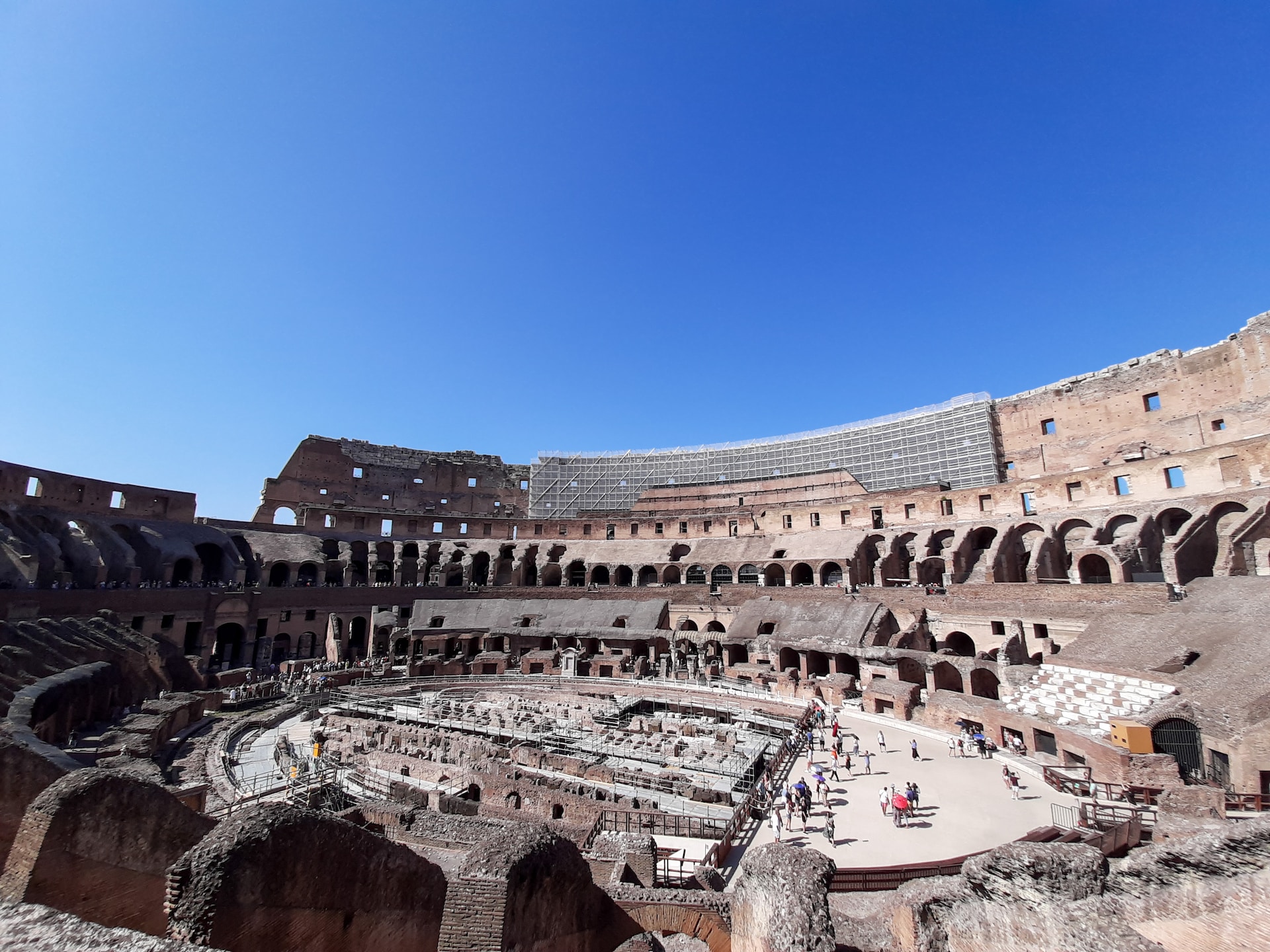}\hfill
		\caption{From left to right, top to bottom: Test images from unsplash.com (artist and size): (1) Kadyn Pierce (4000x6000), (2) Willian Justen de Vasconcellos (3024x4032), (3) Emre Raizyq (6240x4160), (4) Kevin Charit (5780x3853), (5) Vishwajeet Nishad (2560x1920), (6) Christoph Nolte (2400x1920), (7) Lukasz Rawa (960x640), (8) Mark Stuckey (1440x1920) (left to right, top to bottom).}
		\label{fig:meine-grafik}
\end{figure}

\begin{figure}
        \centering
		\includegraphics[scale=0.8]{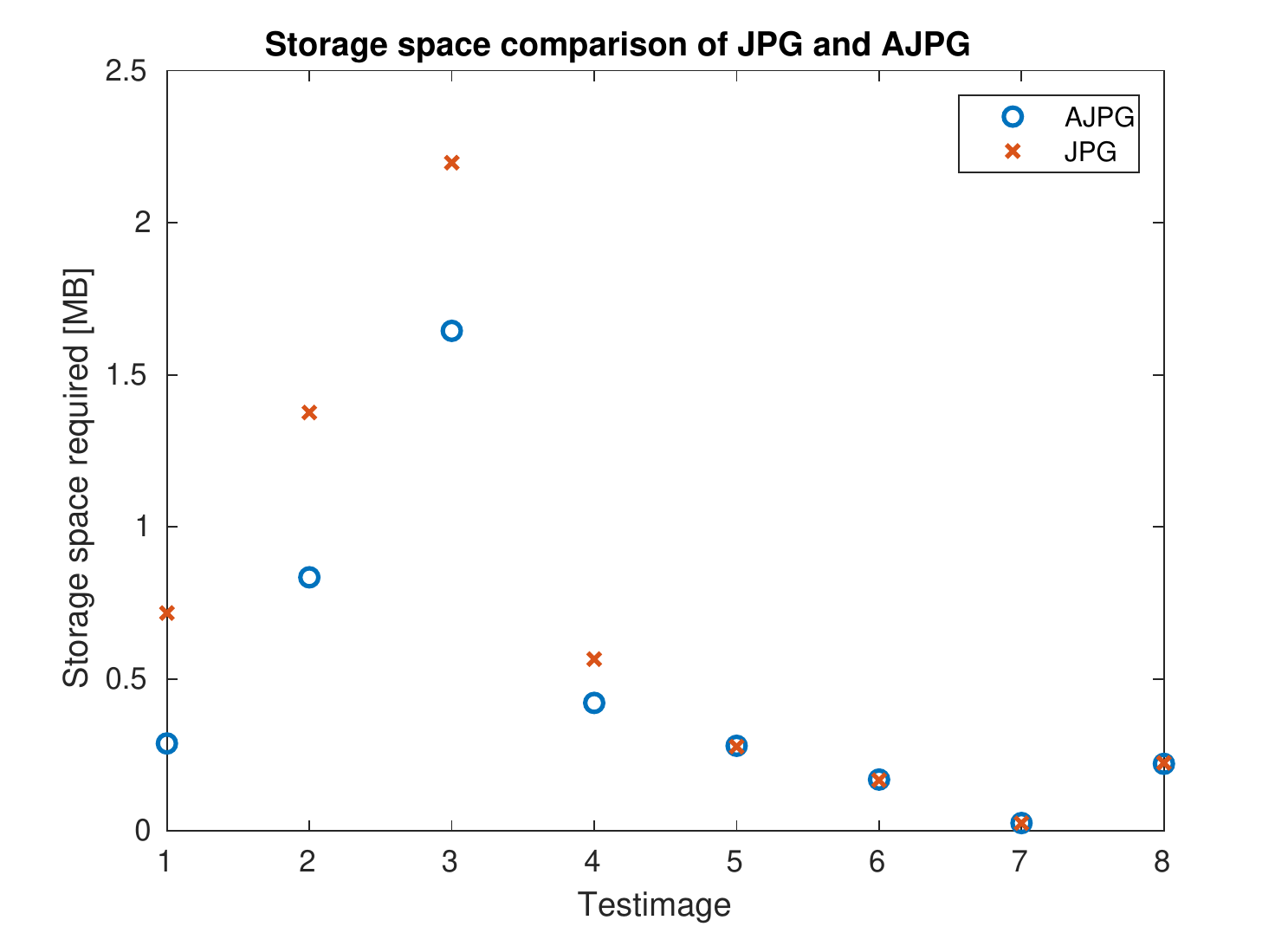}\\
        \includegraphics[height=14mm]{kadyn_pierce.jpg}
		\includegraphics[height=14mm]{willian_justen_de_vasconcellos.jpg}
		\includegraphics[height=14mm]{emre_raizyq.jpg}
  \hspace{2mm}
		\includegraphics[height=14mm]{kevin_charit.jpg}		
  \hspace{2mm}
		\includegraphics[height=14mm]{vishwajeet_nishad.jpg}
  \hspace{2mm}
		\includegraphics[height=14mm]{christoph_nolte.jpg}
  \hspace{1.5mm}
		\includegraphics[height=14mm]{lukasz_rawa.jpg}
  \hspace{1mm}
		\includegraphics[height=14mm]{mark_stuckey.jpg}
  \caption{Storage space required for the original JPEG files (JPG) and their adaptive recompression using Algorithm~\ref{alg:adaptive} (AJPG).}\label{fig:storage}
\end{figure}
 
\begin{figure}\centering
	\includegraphics[width=0.49\textwidth]{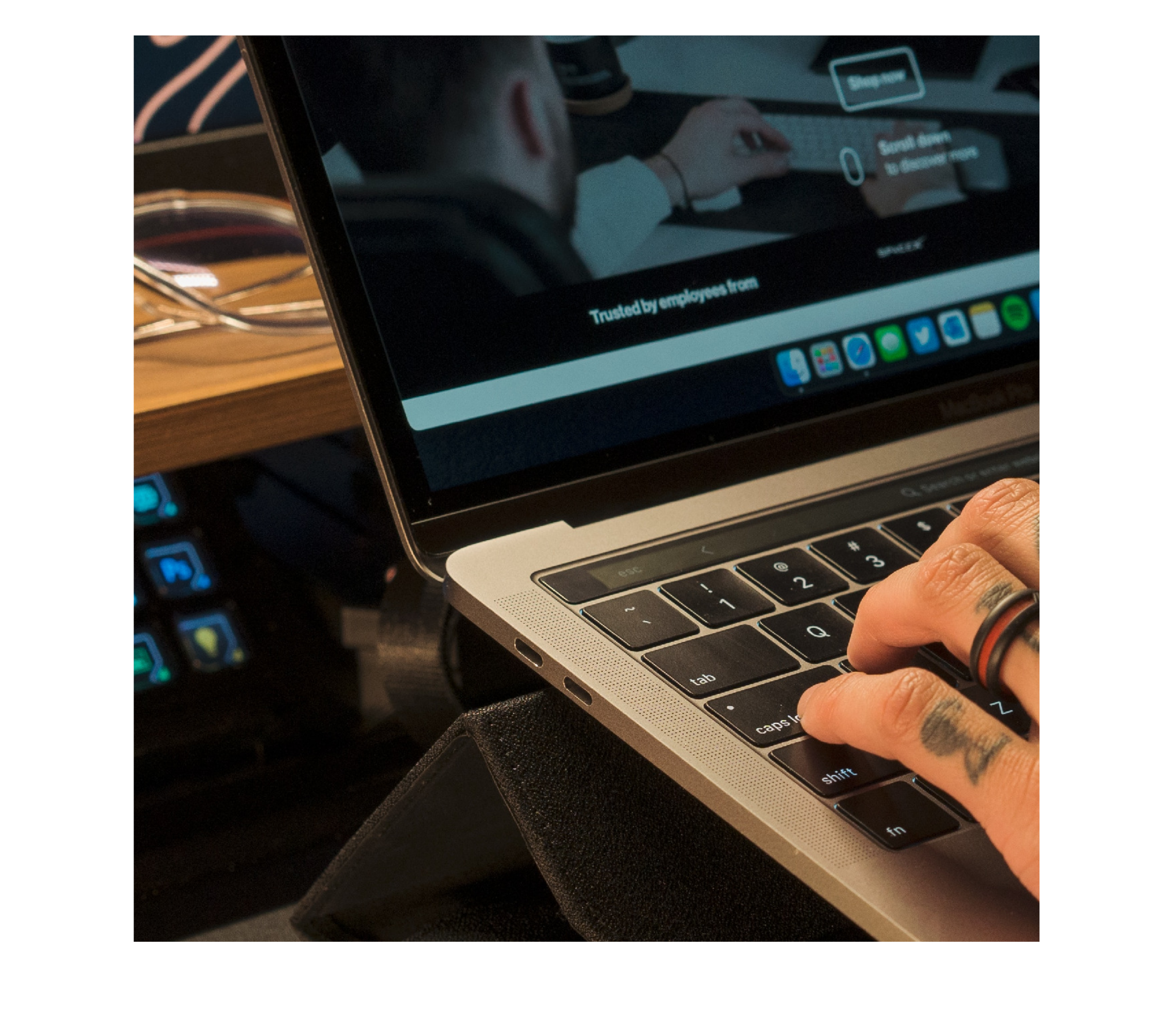} \\        
        \includegraphics[width=0.49\textwidth]{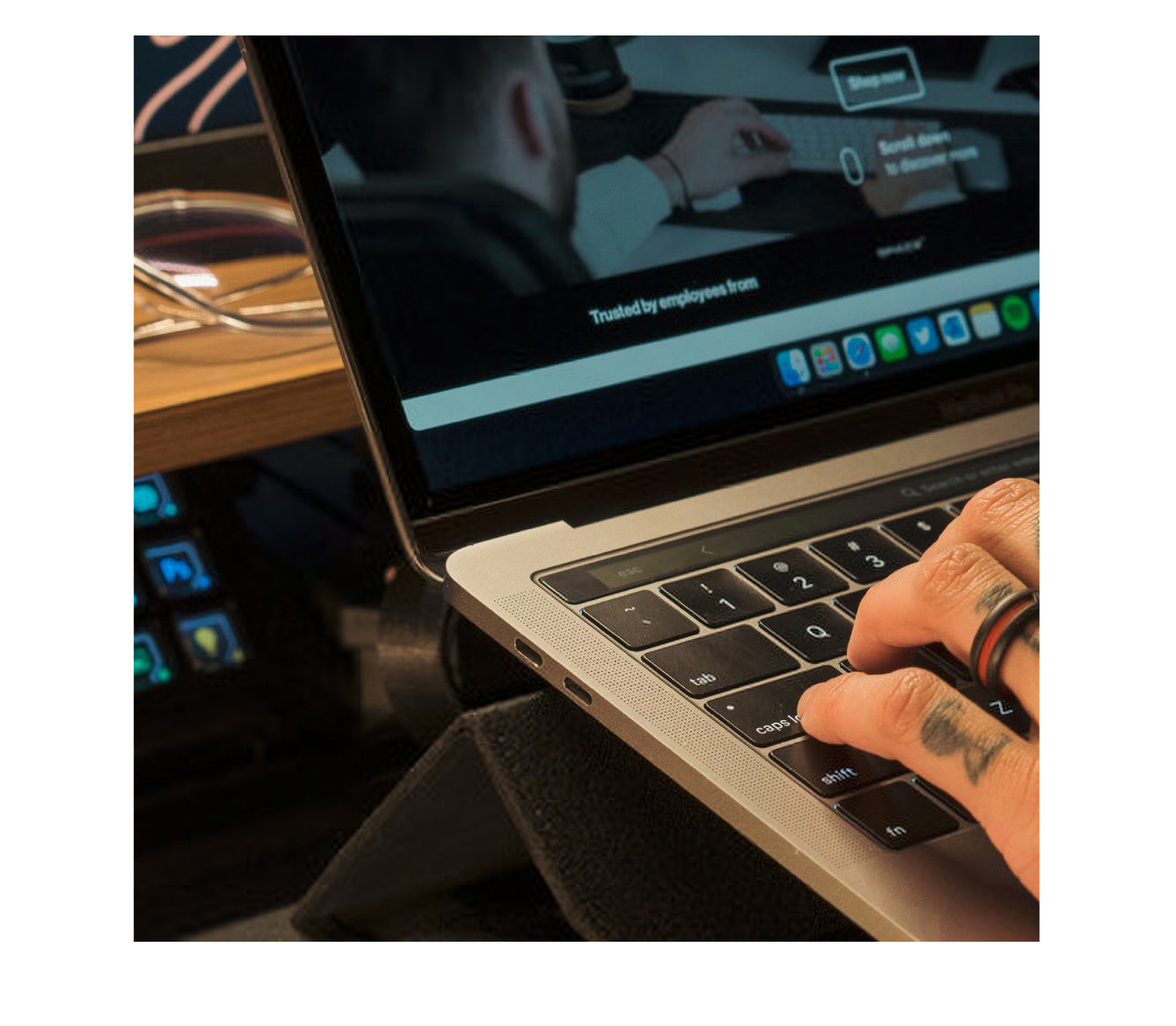}
        \hspace{-10mm}
        \includegraphics[width=0.49\textwidth]{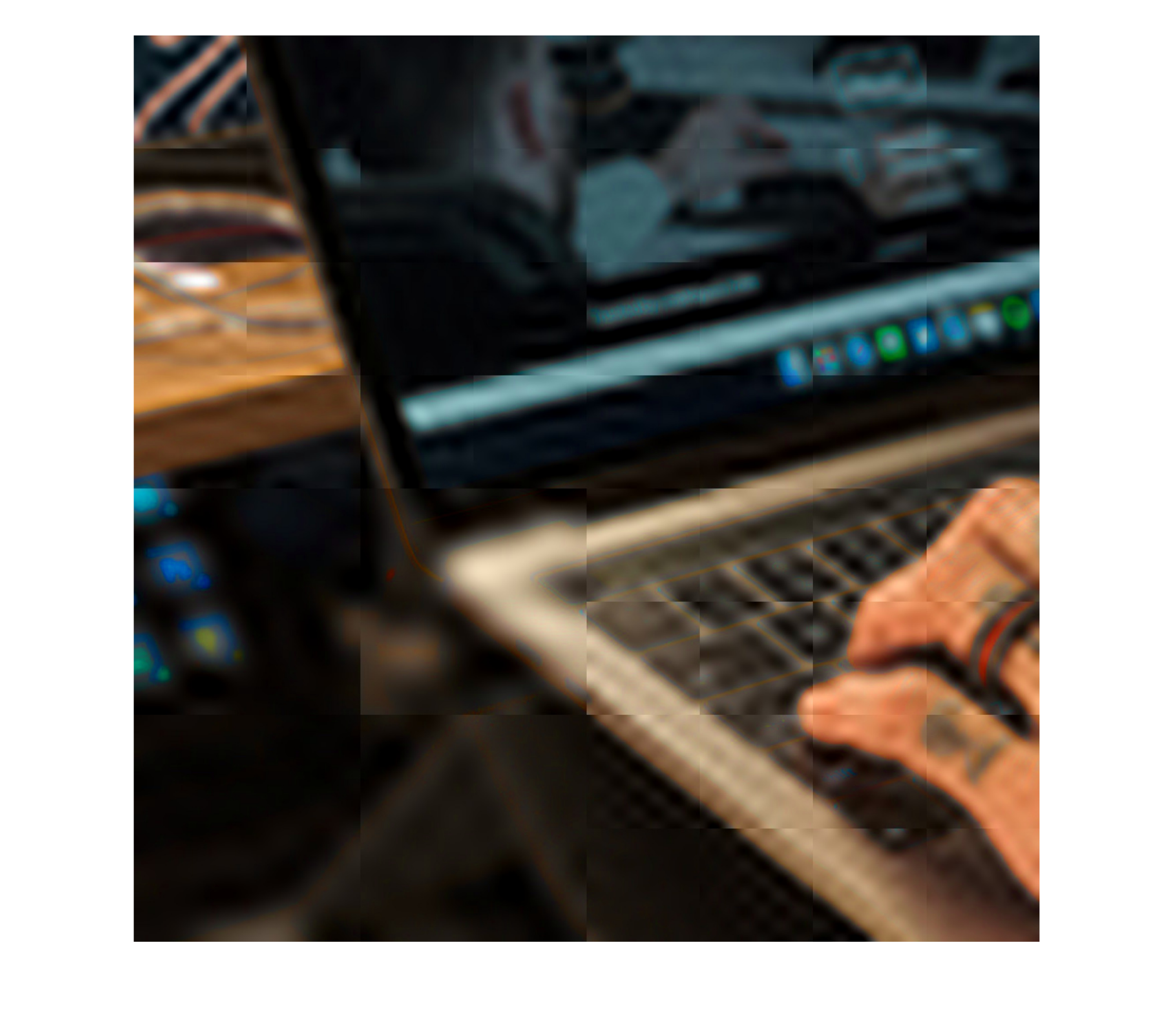}
	\includegraphics[width=0.49\textwidth]{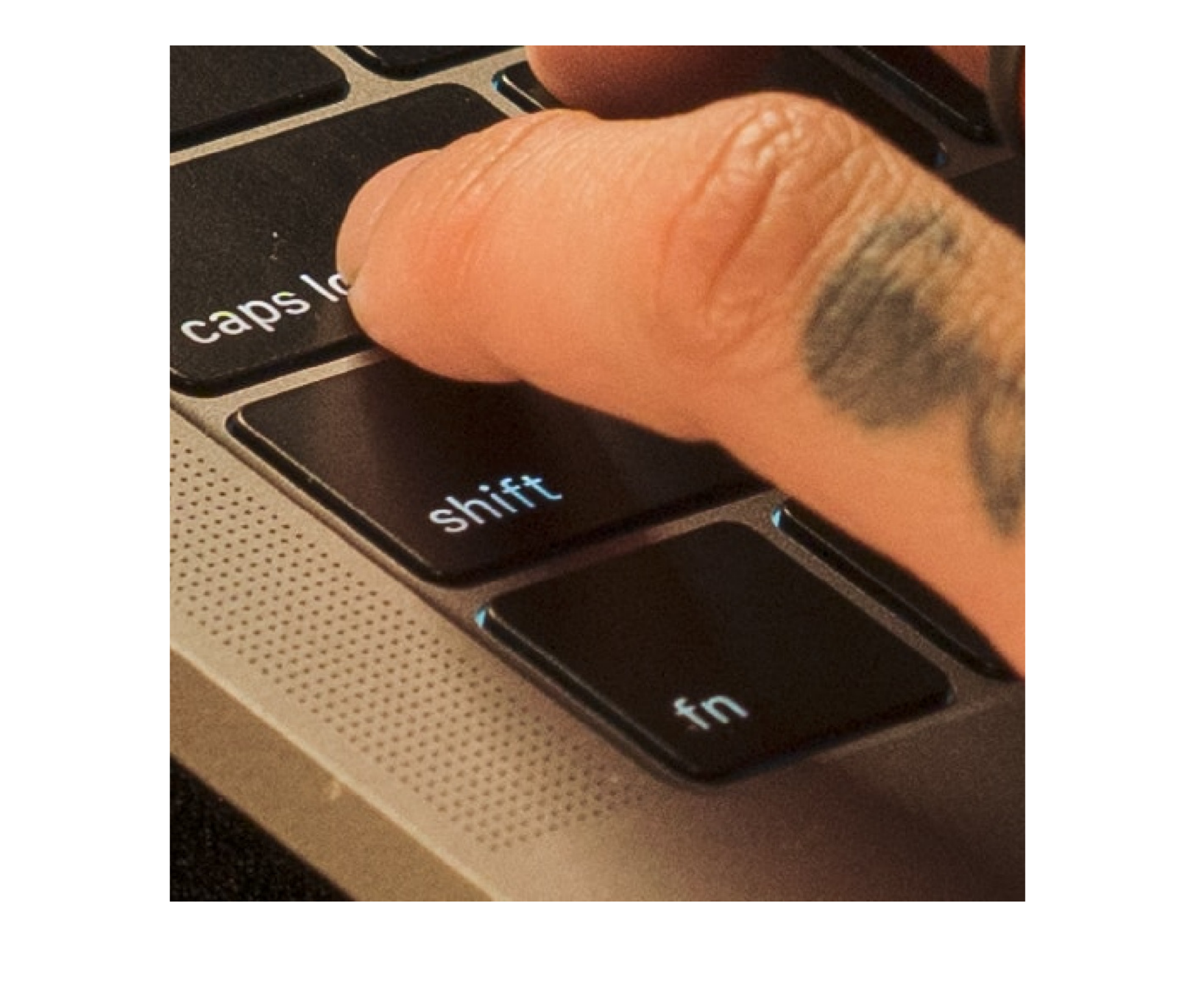}		
        \hspace{-12mm}
        \includegraphics[width=0.49\textwidth]{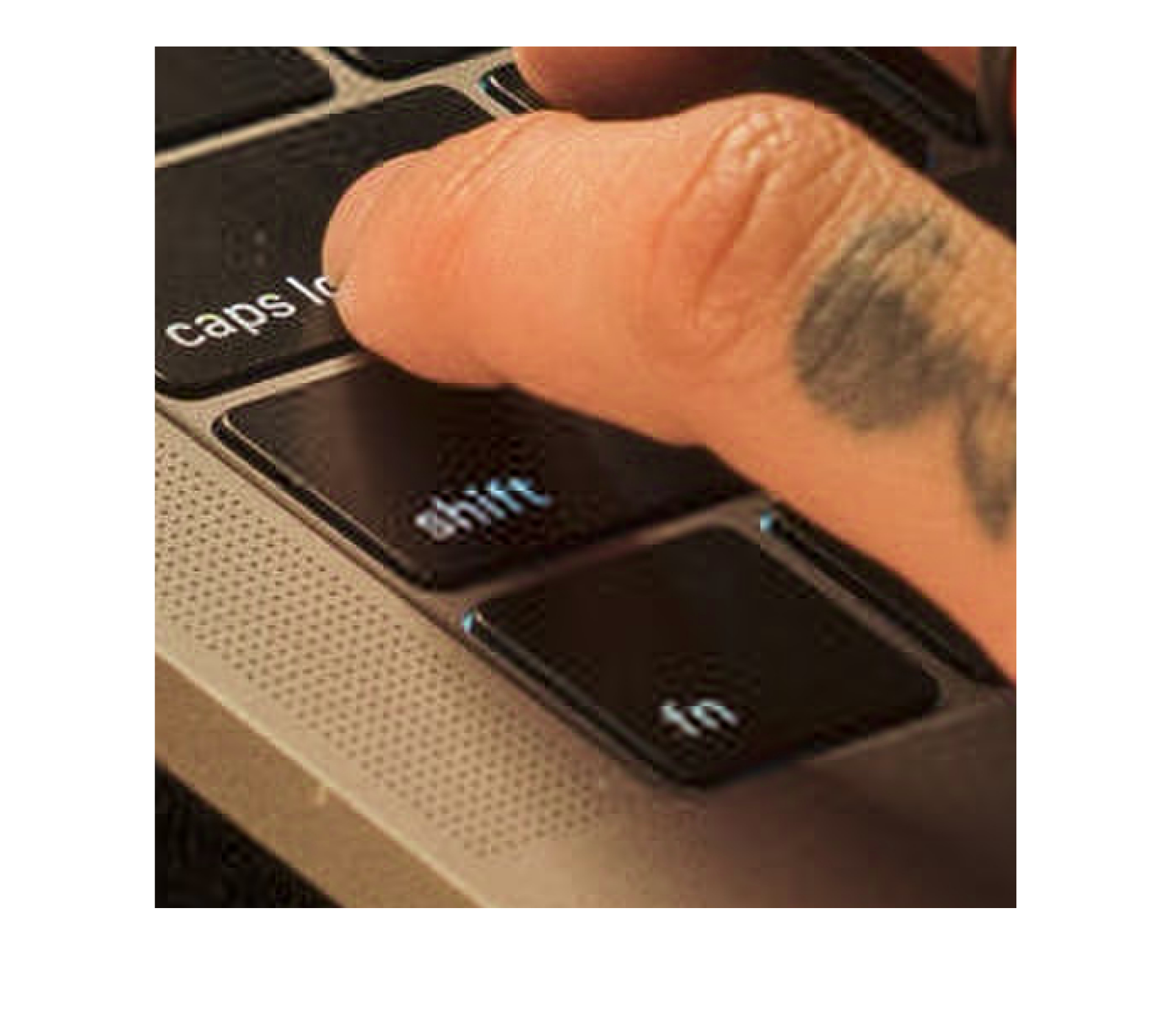}	
		\caption{Comparison of the original image (top) with the two output of the adaptive algorithm (middle) for small and large $\tau$. Closeup of original (bottom-left) and recompression (bottom-right) used in Figure~\ref{fig:storage}. Image from unsplash.com by artist Kadyn Pierce.}\label{fig:qual}
	\end{figure}
 \begin{figure}\centering
		\includegraphics[width=0.49\textwidth]{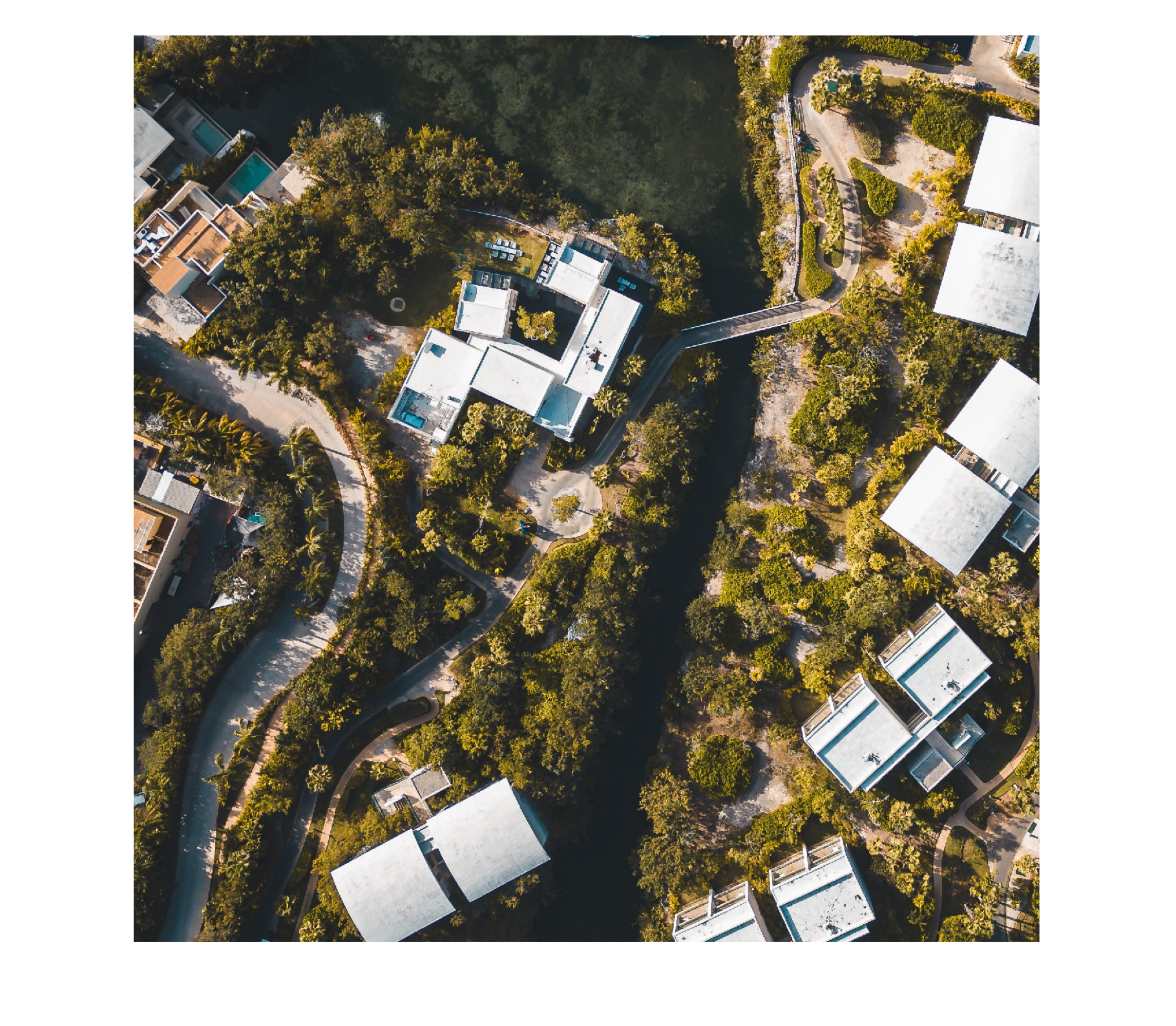} \\
        \includegraphics[width=0.49\textwidth]{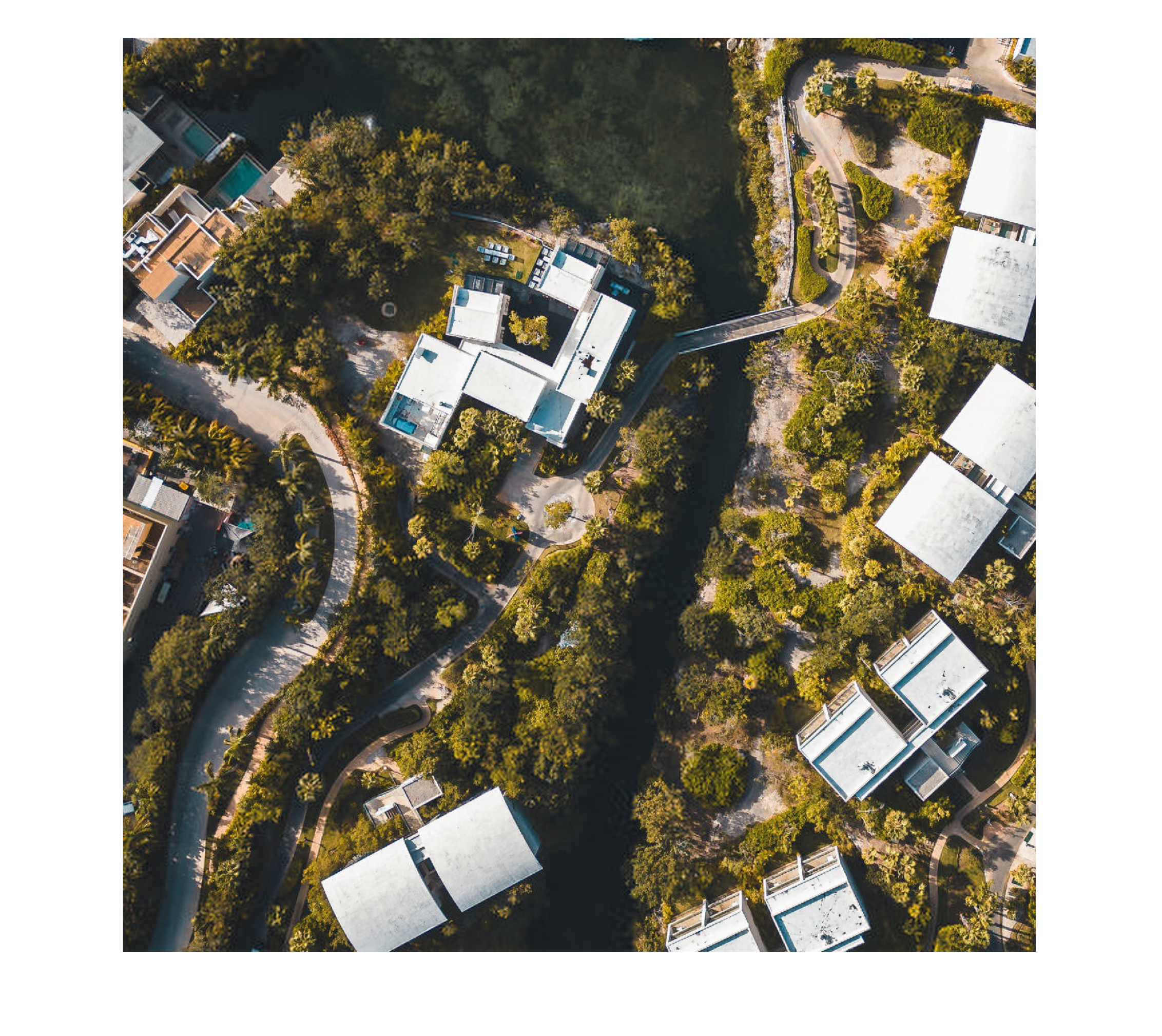}
        \hspace{-10mm}
        \includegraphics[width=0.49\textwidth]{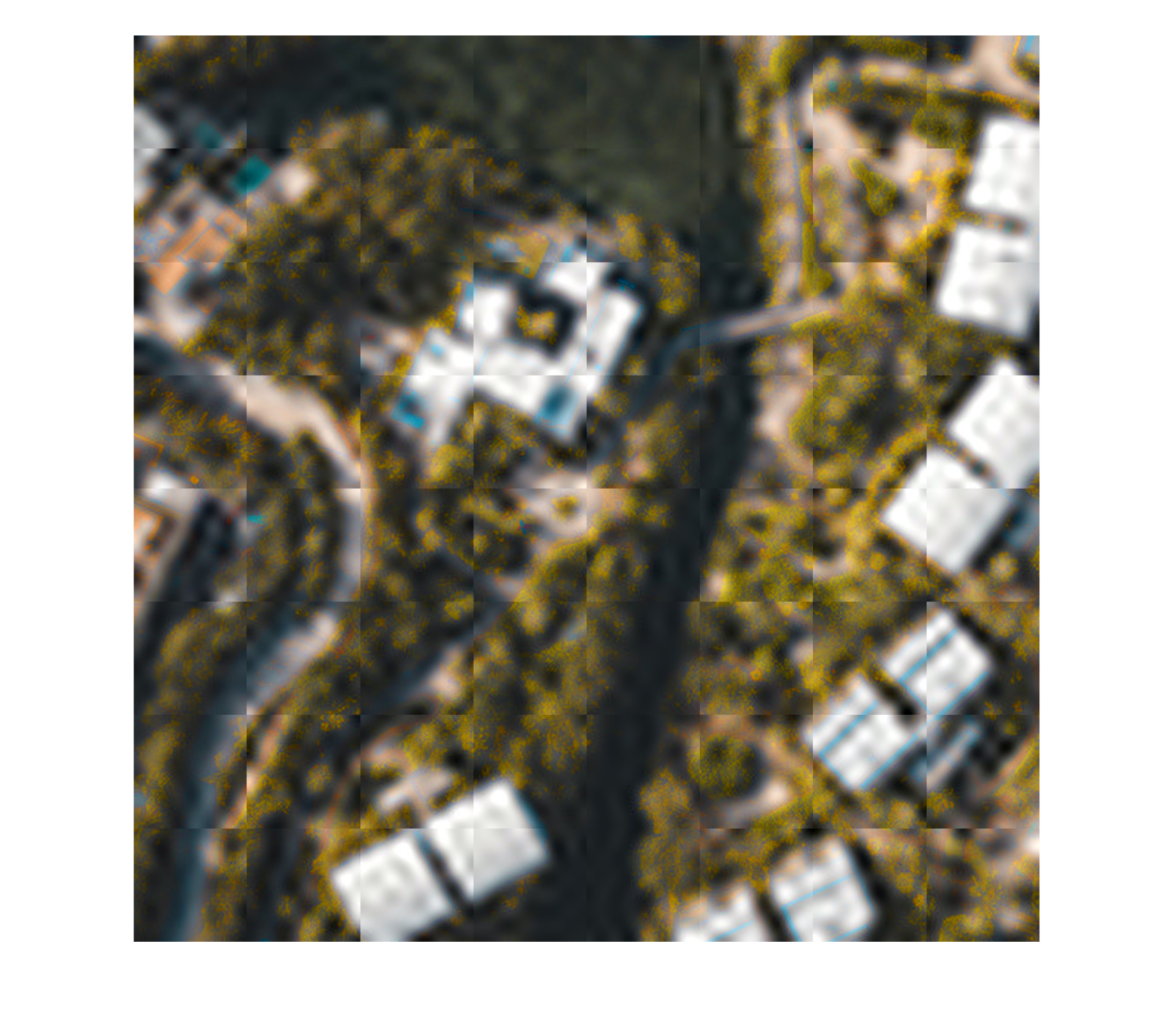}
		\includegraphics[width=0.49\textwidth]{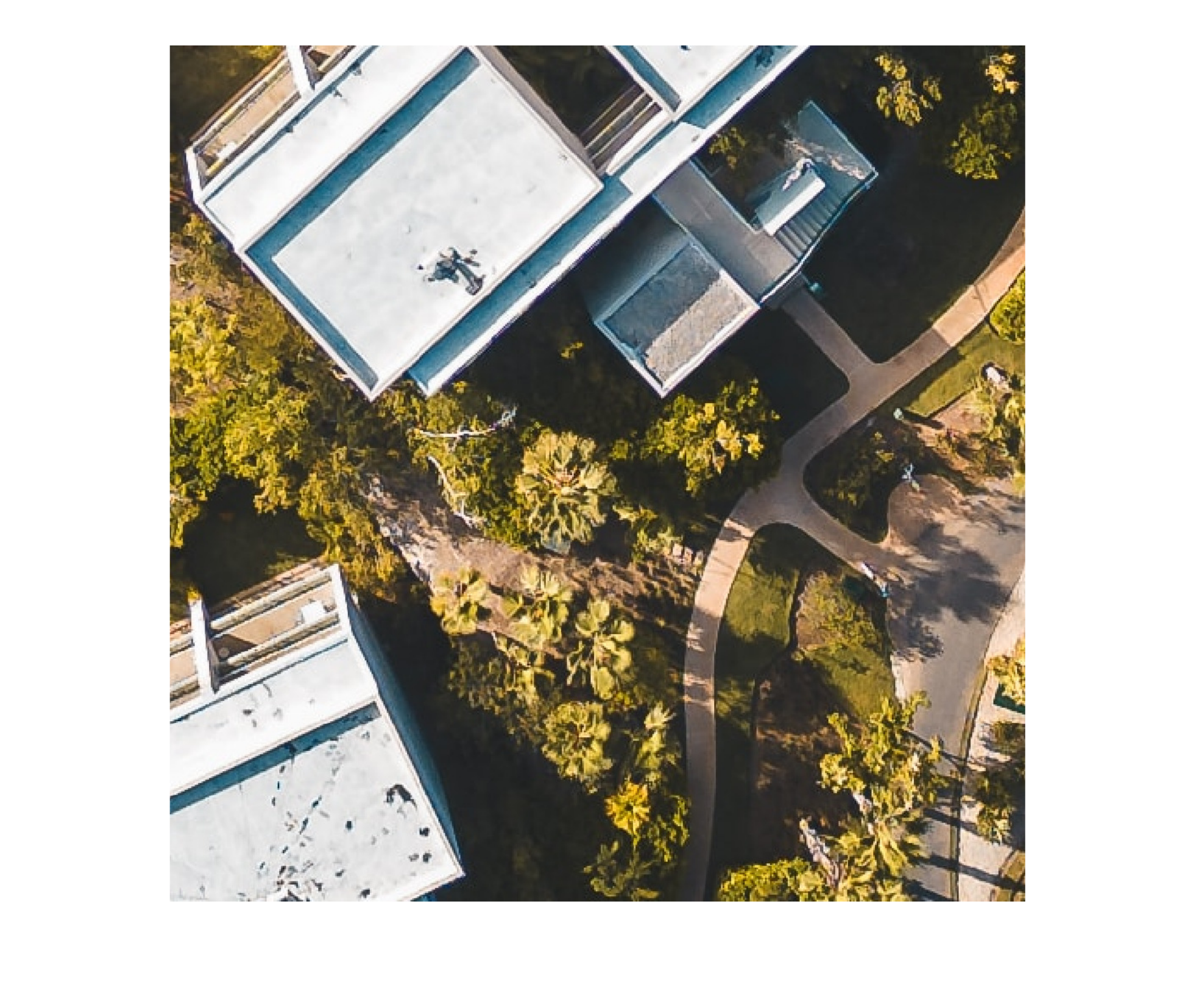}		
        \hspace{-12mm}
        \includegraphics[width=0.49\textwidth]{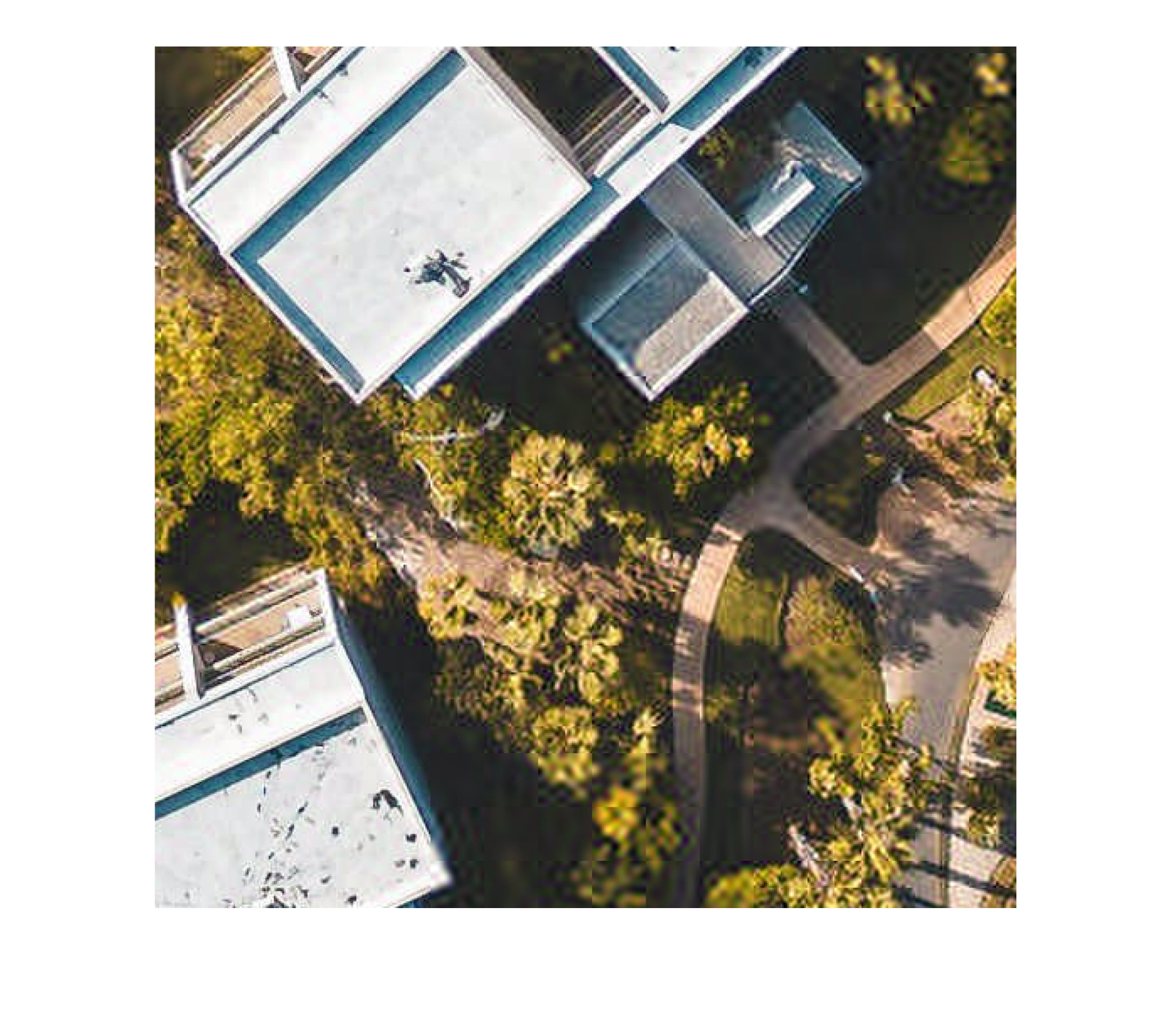}	
		\caption{Comparison of the original image (top) with the two output of the adaptive algorithm (middle) for small and large $\tau$. Closeup of original (bottom-left) and recompression (bottom-right) used in Figure~\ref{fig:storage}. Image from unsplash.com by artist Willian Justen de Vasconcellos.}\label{fig:qual2}
	\end{figure}
 \begin{figure}\centering
	\includegraphics[width=0.49\textwidth]{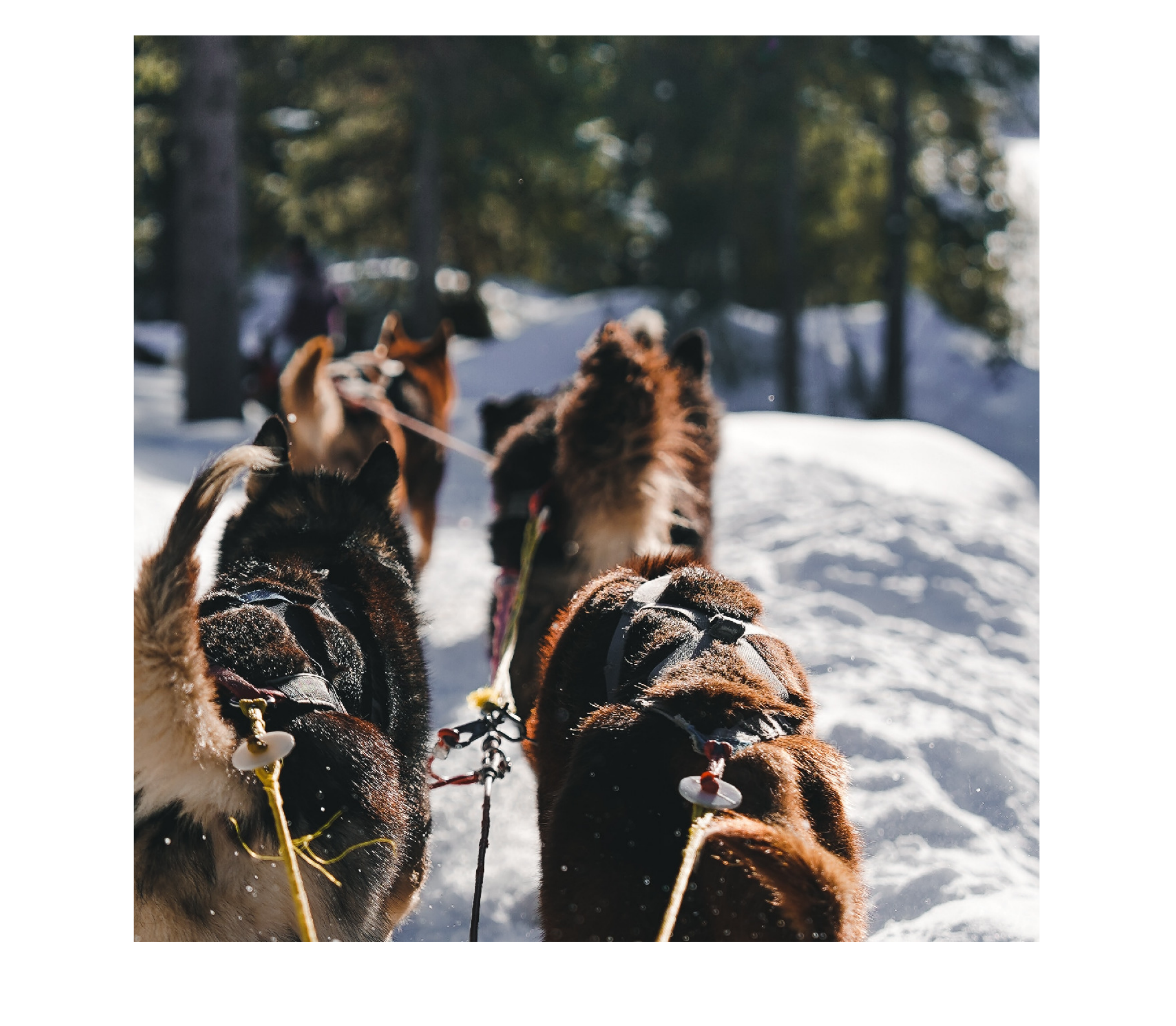} \\
        \includegraphics[width=0.49\textwidth]{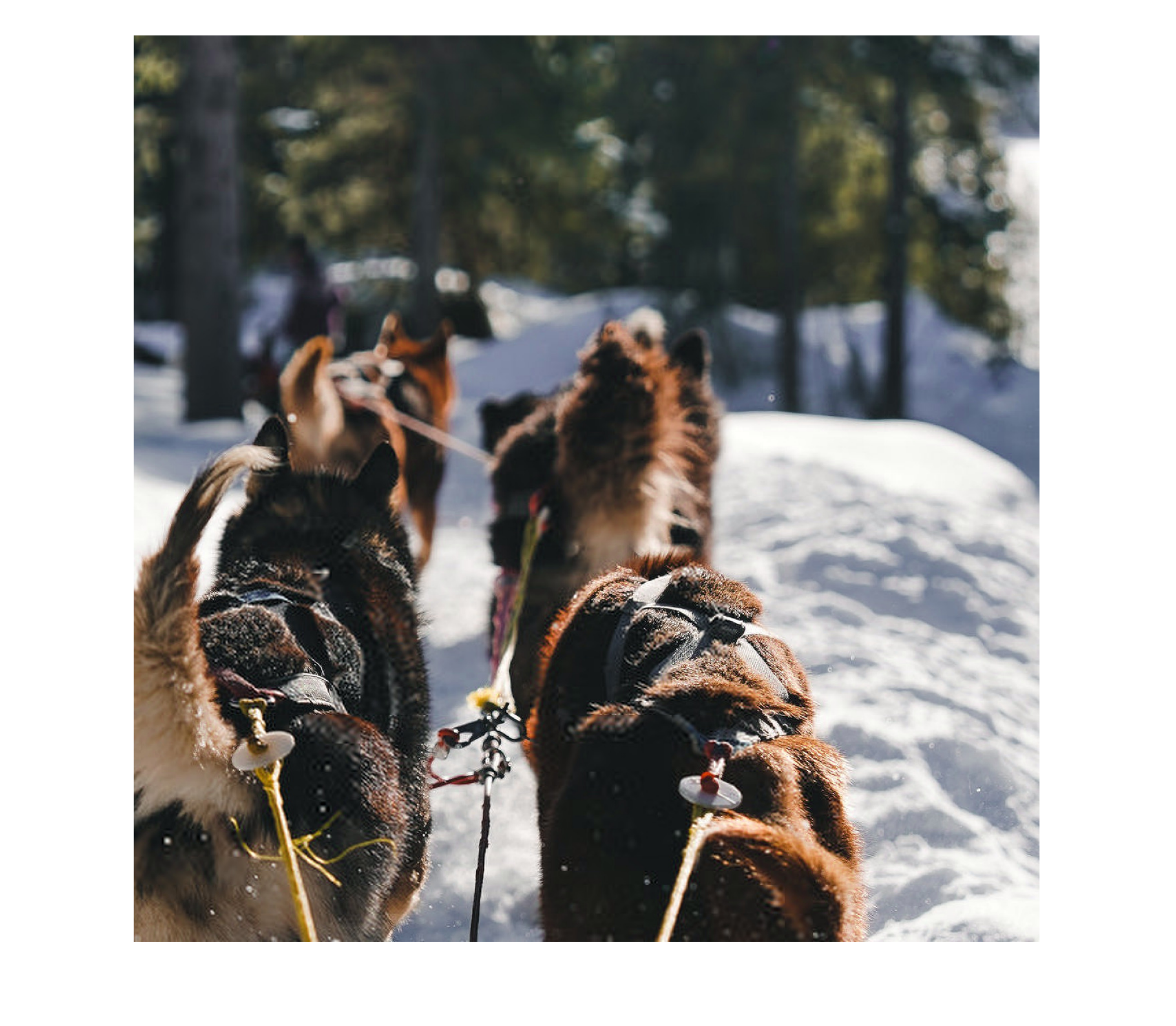}
        \hspace{-10mm}
        \includegraphics[width=0.49\textwidth]{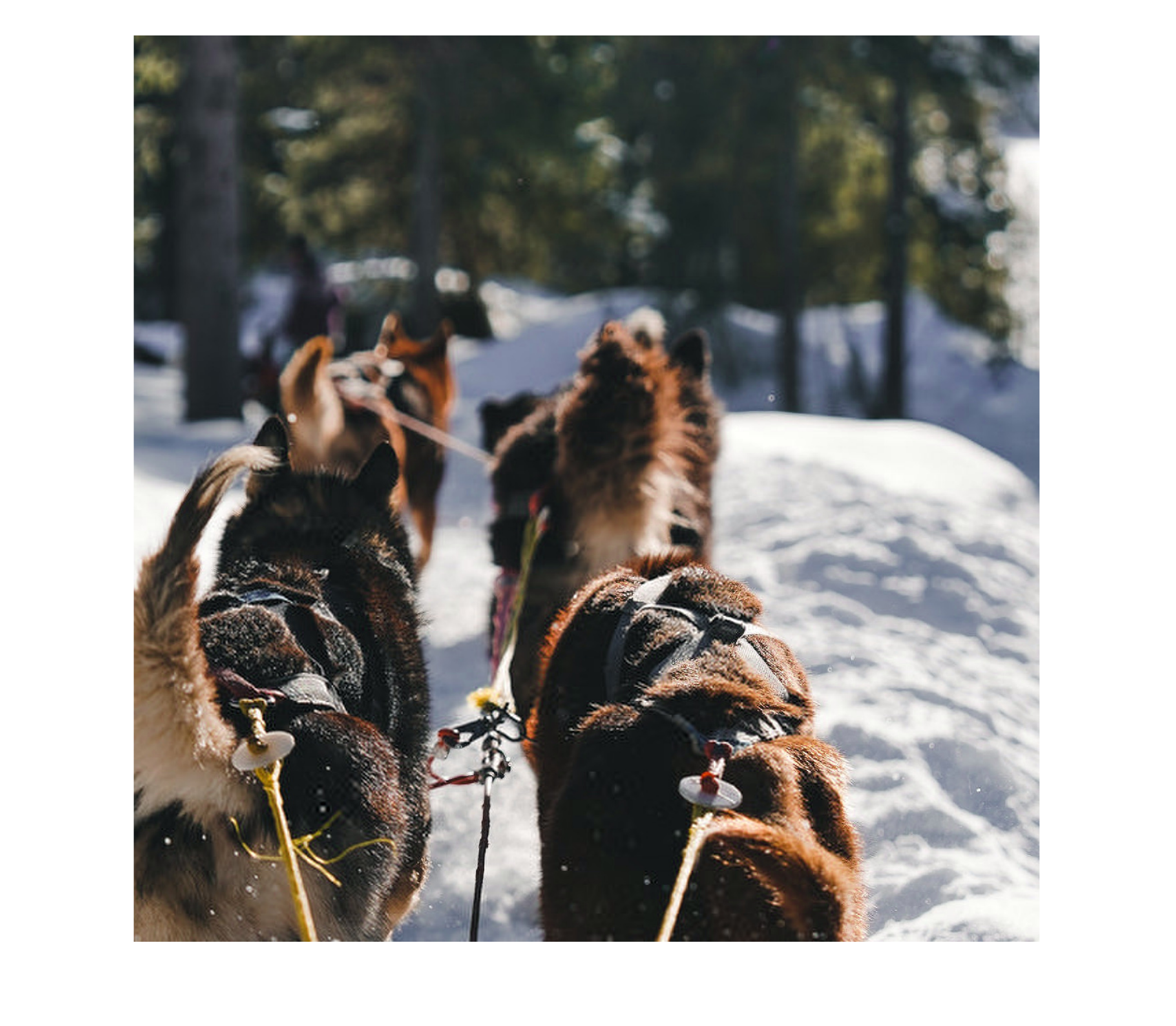}		
        \includegraphics[width=0.49\textwidth]{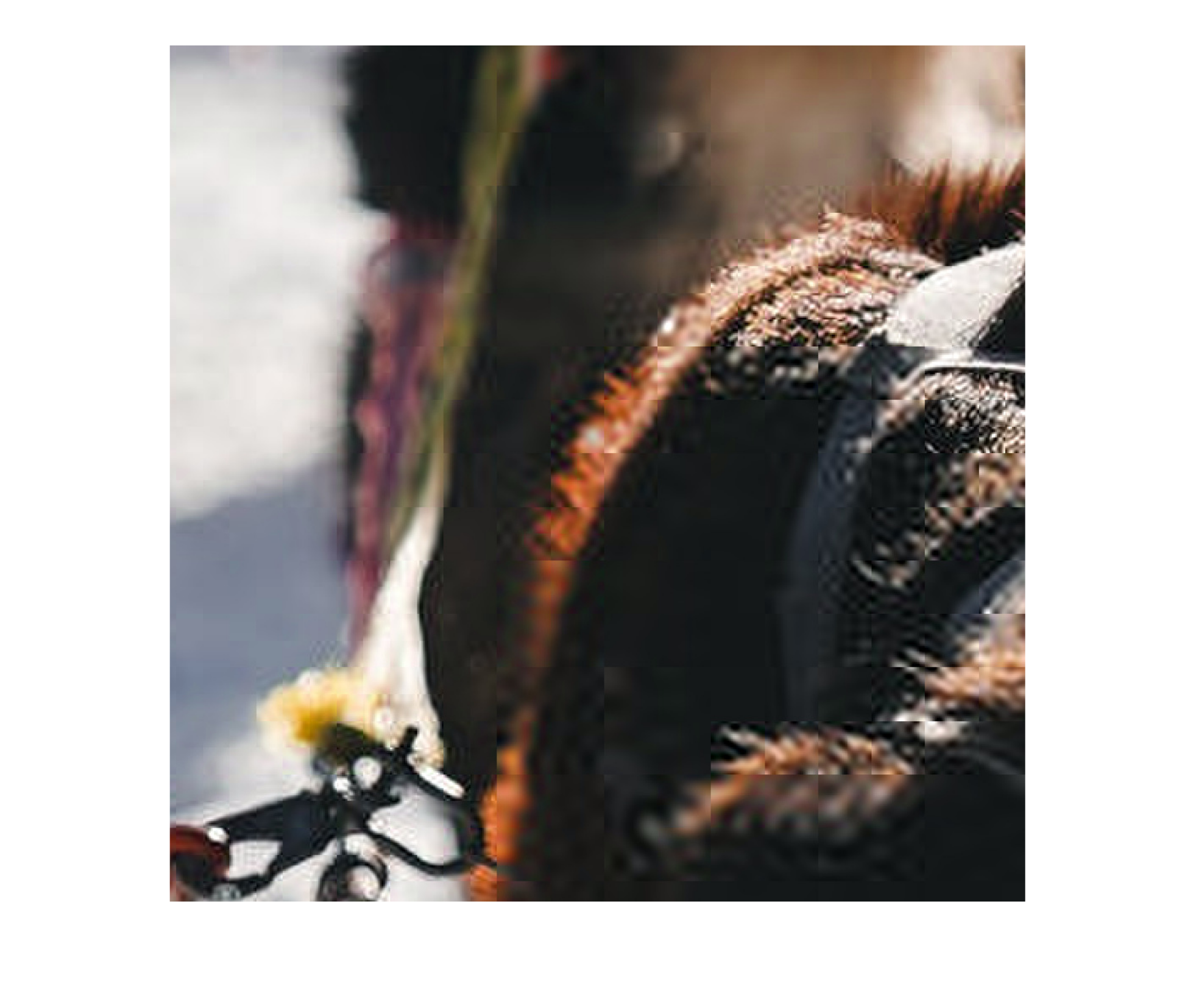}		
        \hspace{-12mm}
        \includegraphics[width=0.49\textwidth]{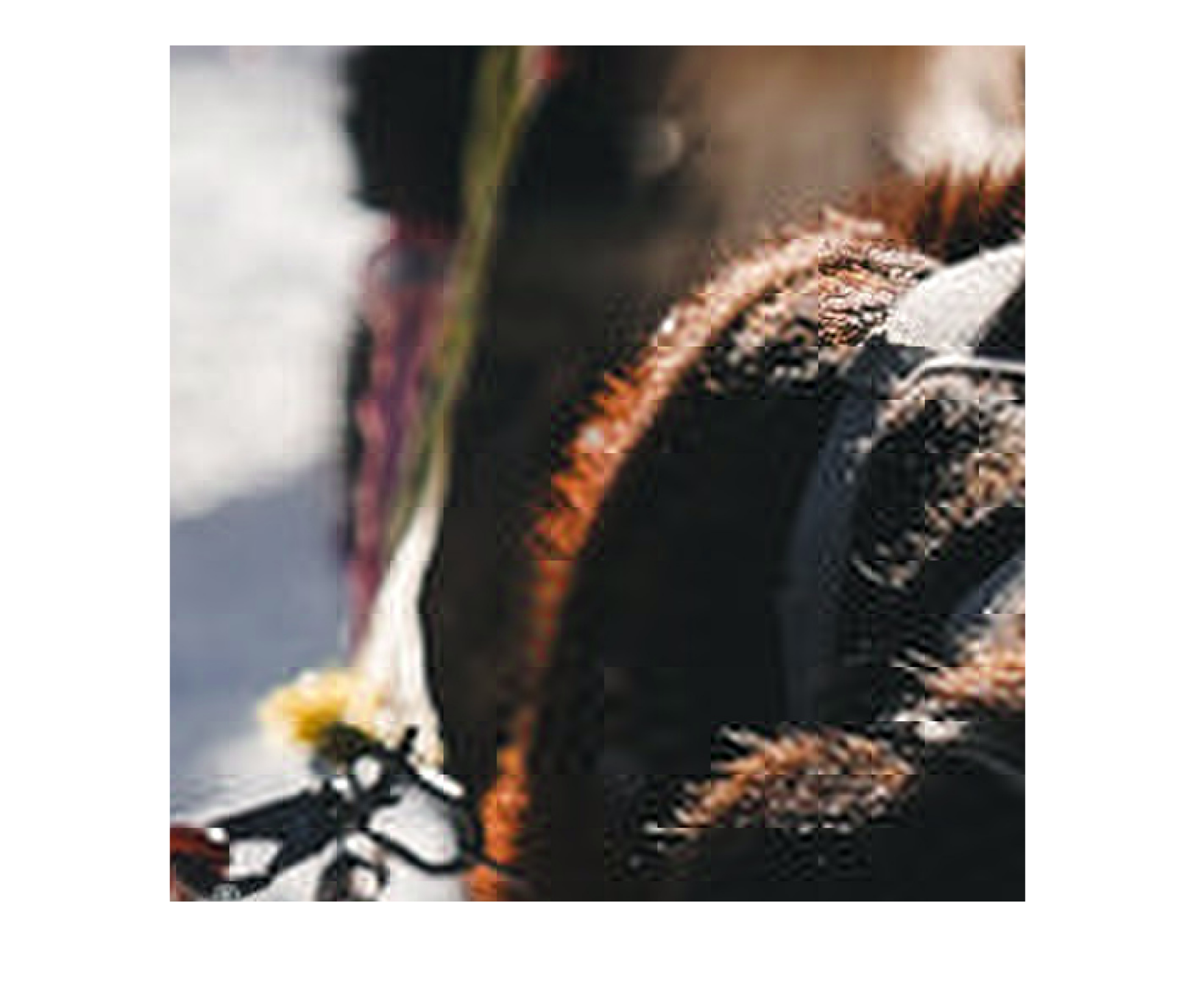}	
		\caption{Comparison of the original image (top) with the output of the adaptive algorithm with respect to the $L^2$ norm (middle-left) and the BV norm (middle-right). On the bottom row a detail of the $L^2$ output (left) and the BV output (right) is compared, which suggests that the BV norm may be better suited. The respective threshold parameters are chosen such that both outputs would take up almost the same storage space. Image from unsplash.com by artist Kevin Charit.}\label{fig:bvl2}
	\end{figure}

\section{Conclusion}
We demonstrated the feasibility of adaptive JPEG compression both in practical experiments and in mathematical optimality results. From an efficiency standpoint, it would make sense to start with a finer uniform grid $\TT_0$. Both Algorithm~\ref{alg:adaptive} and Algorithm~\ref{alg:reconstruct} parallelize in a straightforward fashion over the elements of the initial grid $R\in\TT_0$. 

The method can be extended in several other ways: A first idea is to apply the algorithm to videos instead of images. This would require us to run Algorithm~\ref{alg:adaptive} on a mesh consisting of three dimensional cuboid elements $R$ (each cuboid can be split into eight children) and to approximate $I|_R$ by a discrete 3D cosine transform (or a variant which treats the time coordinate separately).

Another extension of the method applies Binev's $hp$-approximation algorithm~\cite{binev:2018}. Instead of increasing the polynomial degree (as is done in the finite element method), $p$ refinement can modify the quantization matrix locally on each element $R$, or it can alter the number of frequencies that are stored for each $R$.

Finally, it would be interesting to replace the norm $\norm{\cdot}{}$ in the local errors by some more sophisticated method which judges subjective quality loss of images. 
\begin{funding}
  Funded by the Deutsche Forschungsgemeinschaft (DFG, German Research Foundation) -- Project-ID 258734477 -- SFB 1173 as well as the Austrian Science Fund (FWF)
under the special research program Taming complexity in PDE systems (grant SFB F65).
\end{funding}
\bibliographystyle{alpha}
\bibliography{literature}
\end{document}